\documentclass{amsart}
\usepackage{amsmath}
\usepackage{latexsym,amssymb}
\usepackage{color}

\newtheorem{thm}{Theorem}[section]

\newtheorem{lem}[thm]{Lemma}
\newtheorem{pro}[thm]{Proposition}

\newtheorem{cor}[thm]{Corollary}

\newtheorem{eg}[thm]{Example}

\newcommand{\dom}{\text{dom}}
\newcommand{\bea}{\begin{eqnarray*}}
\newcommand{\eea}{\end{eqnarray*}}

\newcommand{\ben}{\begin{enumerate}}
\newcommand{\een}{\end{enumerate}}

\newcommand{\bi}{\begin{itemize}}
\newcommand{\ei}{\end{itemize}}

\newcommand{\ol}{\overline}

\newcommand{\mc}{\mathcal}

\begin{document}

\title{Restriction in program algebra}
\author{Marcel Jackson}
\author{Tim Stokes}


\maketitle

\section{Introduction}  \label{intro}

Let $P(X,Y)$ be the set of all partial functions having domains in the non-empty set $X$ and mapping into non-empty $Y$.  In Berendsen et al.~\cite{minusover}, the operations of \emph{override} and \emph{update} for such partial functions were considered.  Override was defined there as follows: for all $f,g\in P(X,Y)$, 
\[
(f \sqcup g)(x):=\begin{cases} f(x)&\mbox{ if }x\in \dom(f),\\
g(x)&\mbox{ if } x\in \dom(g)\backslash\dom(f),\\
\mbox{ undefined}&\mbox{ otherwise}.
\end{cases}
\]
The operation of update $f\diamond g$ is then defined in \cite{minusover} to be the restriction of $g\sqcup f$ to the domain of $f$.  

We also consider the operations of \emph{domain restriction}, $\circ$ given by: 
\[
(f\circ g)(x)= \begin{cases} g(x)&\mbox{ if }x\in \dom(f),\\
\mbox{undefined}&\mbox{ otherwise},
\end{cases}
\]
\emph{minus}, $-$, given by:
\[
(f-g)(x)=\begin{cases} f(x)&\mbox{ if }x\in \dom(f), x\not\in \dom(g),\\
\mbox{undefined}&\mbox{ otherwise},
\end{cases}
\]
\emph{intersection}, $\cap$:
\[
f\cap g= \{(x,y)\in X\times Y\mid (x,y)\in f\text{ and }(x,y)\in g\}
\]
and \emph{difference}, $\backslash$:
\[
f\backslash g= \{(x,y)\in X\times Y\mid (x,y)\in f\text{ and }(x,y)\not\in g\}.
\]
Before we proceed further, we note that the surrounding literature contains  a number of conflicting notations for the operations just introduced.  
In \cite{minusover},  the operation we have denoted $\circ$ is referred to in an equivalent form as ``intersection", with definition and notation $f@g:=g\circ f$.  The notation for~$\sqcup$ in \cite{minusover} is $\triangleright$, however, the notation $\triangleright$ is well-established for the operation $\circ$ since at least the 1962 work of Vagner \cite{vagner} where it is called restrictive multiplication.  As the present work relates to both~\cite{minusover} and~\cite{vagner}, we have adopted the neutral operation $\circ$ to avoid confusion; this was also done in the second author's work with Hirsch~\cite{H+S}.  We use the term ``domain restriction" (as in Borlido and McLean \cite{borlido}) for $\circ$, since it is more suggestive of the meaning of the operation than ``restrictive multiplication", and because the term ``intersection" of functions has another accepted meaning, namely that used above, and widely encountered since at least the work of Garvac'ki\u{\i}~\cite{garv}.  Our adopted notation $\sqcup$ for override follows \cite{modrest} (contemporaneous with \cite{minusover}), though it is called  \emph{preferential union} there; $\sqcup$ was also used by the authors in \cite{jacstoverup}.  Finally we use the notation $\diamond$ here for update because it is  simpler to specify in signatures; the notation in~\cite{minusover}, and even in the authors' work \cite{jacstoverup}, is $f[g]$, in place of $f\diamond g$.

The algebras of partial functions considered in \cite{minusover} were those closed under override and minus, and a finite equational axiomatisation was given that was shown complete for equations. It was subsequently observed in Cvetko-Vah, Leech and Spinks \cite{CLS} that these algebras are equivalent to a certain type of skew Boolean algebra previously considered by Leech in \cite{LeechSBA}, where a complete axiomatisation of the algebras was given.  The axioms were equivalent to those given in \cite{minusover}, thereby showing those laws to be strongly complete and not just equationally complete.  In the present article, we take strongly complete as the primary notion of completeness, since in general, the isomorphism class of functional models is a quasivariety always, and a variety only sometimes.  When completeness is established for equations only, we refer to \emph{complete with respect to equations}.

Closely related signatures have also been considered.  Thus in \cite{LeechNSL}, Leech shows that algebras of partial functions closed under domain restriction and override may be axiomatised as certain types of skew lattices; see $3.7$ there, which explains how to view Theorem $3.2$ there as providing such an axiomatisation.    
 Amongst signatures not including override, that consisting of domain restriction and difference has recently been considered in \cite{borlido}, where a finite equational axiomatisation is given.

Note that a number of these operations may be defined in terms of the others; for example: 
\bi
\item $f\circ g=g-(g-f)$ (as noted in \cite{minusover});
\item $f\diamond g=f\circ(g\sqcup f)$ (also as noted in \cite{minusover});
\item $f\circ g=(f\diamond g)\cap g$;
\item $f\cap g=f\backslash(g\backslash f)$;
\item $f-g=f\backslash (g\circ f)$;
\item $f\backslash g=f-(f\cap g)$.
\ei
Hence $\circ$ is expressible in any signature containing both $\diamond$ and $\cap$, and $\diamond$ is expressible in any signature containing both $\circ$ and $\sqcup$.  So the following signatures are equivalent in expressive power:
$$(\diamond,\sqcup,\cap),\ (\circ,\sqcup,\cap),$$
and similarly so are 
$$(\backslash,\circ),\ (-,\cap)$$
as well as
$$(\backslash,\sqcup),\ (-,\cap,\sqcup).$$
In every case above, $\circ$ is expressible.  Other less expressive cases we consider in which $\circ$ is either in the signature or definable within it are as follows:
$$(-),\ (-,\diamond),\ (\circ), (\circ,\cap),\ (\diamond,\cap).$$

The signature $(\backslash,\sqcup)$ can be viewed as a kind of ``master signature", since all the operations considered here can be expressed within it.  Written directly, we have
$$f-g=(g\sqcup f)\backslash g,\ f\cap g=f\backslash (f\backslash g),\ f\circ g= g\backslash ((f\sqcup g)\backslash f),\ f\diamond g=(g\sqcup f)\backslash [(f\sqcup g)\backslash f].$$
This signature is axiomatised by Cirulis \cite{cirulis} (see Theorem $5.6$ there), an alternative description in terms of skew Boolean intersection algebras being given in \cite{CLS}, and we give yet another description of it in Section \ref{sec:axioms2}.  A further axiomatisation is given by Bignall and Leech~\cite{bigleech}, where the class is shown to be the variety generated by the pointed discriminator algebras, which is the class of pointed comparison algebras as defined in Stokes~\cite{compsemi}.  

However, aside from $(-,\sqcup)$ and $(\backslash,\circ)$, weaker signatures such as $(\circ,\cap,\sqcup)$ have mostly not been axiomatised; as is pointed out in \cite{CLS}, this latter signature gives rise to a certain class of skew lattices with intersection.  In this article we axiomatise this case, but also recover known results such as those in \cite{borlido} and \cite{cirulis}, and obtain new results for a large range of combinations of the operations previously discussed.  In addition, our methods make possible the modelling of {\em program concatenation}, corresponding to composition of partial functions; in this case, we assume $Y=X$.  Indeed our approach is to model the case with composition first, then recover the composition-free case for $P(X,Y)$ from it.

For two equivalent signatures, it suffices to finitely axiomatise one in order to establish a finite axiomatisation of the other: simply translate the needed laws using the above relations, and add in as additional laws the relations used in the translation.  All the axiomatisations obtained in what follows are finite, and most are equational, though some are properly quasiequational.  Since $\circ$ is expressible in every signature, we develop an approach based on it.  

The paper is organised as follows.  In Section \ref{sec:axioms} to follow, we introduce each of the classes of algebras considered, giving axioms for them which we claim are sound in each case, and in later sections show are complete.  The signatures come in pairs: one without composition and one with, the latter obtained from the former by the addition of some extra axioms involving composition.  Following that, in Section \ref{sec:filters}, we develop some important preliminary facts involving certain fundamental quasiorders and the associated filters as well as certain equivalence relations, all of which is done at the ``composition-free" level.  Then in Section \ref{sec:rep}, we establish completeness of the axioms given in Section \ref{sec:axioms} by using a single unified approach to represent algebras in each class, applying it to each in turn to show each operation in the relevant signature is correctly represented.  In Section \ref{sec:axioms2}, we tidy up some of the earlier axiomatisations, showing that some quasivarieties are finitely based varieties, or are not varieties at all.  We then present a table in which the axiomatisation status of a wide range of signatures is given.  We conclude the main body of the article with some open questions in Section \ref{sec:open}.  Included after that is some supplementary material in which simplifications of our various axiom sets are obtained using the {\em Prover9/Mace4} software \cite{P9M4}.

Throughout, for a given signature consisting of some of the operations just described, we use the term ``algebra of partial functions" of that signature to refer to a set of partial functions in $P(X,Y)$ (or $PT(X)$ if  composition is part of the signature) closed under the operations in the signature.  When we speak of the algebras of partial functions of some given signature as being axiomatised as a class of axiomatically defined algebraic structures, we of course mean that the isomorphism class of such algebras is that given class.

\section{The axioms}  \label{sec:axioms}

\subsection{Right normal bands and 1-stacks}

Recall that a {\em right normal band} $(A,\circ)$ is a semigroup satisfying the following: for all $x,y,z\in A$,
\bi
\item $x\circ x=x$ (so $(A,\circ)$ is a band), and
\item $(x\circ y)\circ z=(y\circ x)\circ z$.
\ei 
We say the right normal band has a {\em zero} if it has a zero as a semigroup; sufficient for this is that $0\circ x=0$ for all $x$, since if this holds then $x\circ 0=x\circ (0\circ x)=(x\circ 0)\circ x=(0\circ x)\circ x=0\circ x=0$.  We can append a new element $0$ to any right normal band $(A,\circ)$, defining $s0=0s=0$ for all $s\in A$ to obtain a semigroup with zero which is clearly a right normal band with zero; denote this by $A^0$.

It is easy to  verify that $P(X,Y)$ is a right normal band under the domain restriction operation~$\circ$ (and indeed it has zero the empty function), whence so is any subsemigroup of it; we call such examples that lie inside $P(X,Y)$ in this way {\em functional}.  It was shown in Vagner~\cite{vagner} that every right normal band is isomorphic to a functional one.  A proof of this is quite straightforward, and we shortly give a proof which also applies to the richer structures we are interested in here.

%

The domain restriction operation $\circ$ is defined for partial functions in $P(X,Y)$.  To model composition $\cdot$ as well, we assume $Y=X$.  It is easy to check that the following laws hold on $PT(X)=P(X,X)$ equipped with domain restriction and composition (where we write $\cdot$ as concatenation):
\ben
\item  the right normal band laws for $\circ$;
\item  associativity of composition;
\item  for all $a,b,c$, $a\circ (bc)=(a\circ b)c$, $a(b\circ c)=(ab)\circ (ac)$.
\een

These laws define the class of {\em 1-stacks}.  The particular axioms just given appear in Schein~\cite{schein}, where it is noted that they axiomatise algebras of partial functions under composition and domain restriction.

We say the 1-stack $(A,\circ,\cdot)$ {\em has zero} if there is $0\in A$ which is a zero with respect to both operations.

\begin{pro}  \label{stackembed0}  
Every right normal band with zero $(A,\circ)$ can be made into a 1-stack with zero by setting $st=0$ for all $s,t\in A$.
\end{pro}
\begin{proof}
With the above definition of $\cdot$, an easy case analysis verifies that the two identities appearing in (3) above both hold, and $0$ is obviously a zero with respect to $\cdot$. 
\end{proof}

\begin{cor}  \label{stackembed}  
Every right normal band $(A,\circ)$ embeds in the right normal band reduct of a 1-stack with zero.
\end{cor}
\begin{proof}
Append $0$ to $(A,\circ)$ to form the right normal band with zero $A^0$ and then apply the previous result.
\end{proof}

From now on, if it is part of the signature, the semigroup operation intended to model composition will take notational precedence over all other operations.  For example the law $a(b\circ c)=(ab)\circ (ac)$ can be written unambiguously as $a(b\circ c)=ab\circ ac$, and so on.

\subsection{Adding intersection}  \label{cap}

Algebras of partial functions closed under intersection are abstractly nothing but semilattices.  But intersection of partial functions is also relatively easily modelled within the setting of right normal bands.  Let us say that an algebra $(A,\circ,\cap)$ equipped with two binary operations is a {\em right normal band with intersection} if $(A,\circ)$ is a right normal band, $(A,\cap)$ is a semilattice, and for all $x,y,z\in A$:
\bi
\item $(x\cap y)\circ x=x\cap y$;
\item $x\circ (y\cap z)=(x\circ y)\cap z$.
\ei

Again, it is routine to check that $P(X,Y)$ is a right normal band with intersection when equipped with domain restriction and intersection.   

Full left distributivity is an easy consequence of the right normal band with intersection laws:
\[
(x\circ y)\cap (x\circ z)=x\circ (y\cap (x\circ z))=x\circ ((x\circ z)\cap y)=x\circ(x\circ (z\cap y))=x\circ (y\cap z).
\]

We introduce composition to the signature.  We say $(A,\cdot,\circ,\cap)$ is a {\em 1-stack with intersection} if 
\ben
\item $(A,\cdot,\circ)$ is a 1-stack,
\item $(A,\circ,\cap)$ is a right normal band with intersection, and
\item for all $s,t,u\in A$, $s(t\cap u)=st\cap su$.
\een

The final law is known to be sound for the functional signature $(\cdot,\cap)$  (see \cite{garv}).  Hence any algebra of functions of signature $(\cdot,\circ,\cap)$ is certainly a 1-stack with intersection.

\begin{pro}  \label{capembed}
Every right normal band with intersection embeds as a reduct in a 1-stack with intersection.
\end{pro}
\begin{proof}
Given the right normal band with intersection $(A,\circ,\cap)$, adjoin a zero element to give $(A^0,\circ)$, giving a 1-stack with all products zero as in the proof of Proposition \ref{stackembed0}. Also extend the definition of $\cap$ to $A'$ by defining $s\cap 0=0=0\cap s$ for all $s\in A'$.  Using case analyses, the result is easily seen to be a 1-stack with intersection in which $(A,\circ,\cap)$ embeds as a subreduct.
\end{proof}

For each of the sets of axioms to follow involving minus, given in Subsections \ref{subsec:minus}, \ref{sub:minusover} and~\ref{subsec:minusup}, we add intersection to the signature by simply adding in the 1-stack with intersection or right normal band with intersection laws, and we refer to such an enhanced algebra as ``with intersection" in all such cases.

\subsection{Adding minus}\label{subsec:minus}

We say $(A,\circ,-)$ is a {\em minus-algebra} if the following laws hold:
\bi
\item $x\circ y=y-(y-x)$;
\item $(A,\circ)$ is a right normal band;
\item $x-x=0$ (this effectively defines $0$);
\item $x\circ 0=0\circ x=0$;
\item $(x-y)\circ x=x-y$;
\item $(x-y)\circ y=0$;
\item $(x-y)\circ z=(x\circ z)-y$;
\item $s-x=t-x\And x\circ s =x\circ t\Rightarrow s=t$.
\ei
(As in other cases to follow in this section, we return to the question of more elegant axioms later.)  
Fix the minus-algebra $(A,-)$ for the remainder of this section.

As in previous cases, we can add composition to the signature.  We say $(A,\cdot,-)$ is a {\em minus-semigroup} if
\ben
\item $(A,-)$ is a minus-algebra,
\item $(A,\cdot)$ is a semigroup,
\item $s(t-u)=st-su$ for all $s,t,u\in A$.
\een

All of these laws are functionally sound, and most follow easily from Figure 3 in \cite{minusover}.  In particular, the final quasiequational law for minus-algebras states that two functions $s,t$ that agree both on the domain of $x$ and outside of the domain of $x$ must be equal.  Law (3) above involving composition can easily be checked by case analysis (both sides being subsets of $st$).

Since $0$ is already in the signature of a minus-algebra, we do not need to adjoin one.

\begin{pro} \label{minusnice}
Every minus-algebra is a minus-semigroup if we define $st=0$ for all $s,t$.
\end{pro}
\begin{proof} 
Let $(A,-)$ be a minus-algebra.  
Defining $s\cdot t=0$ for all $s,t\in A$ gives a semigroup $(A,\cdot)$, and because $0-0=0$, $(A,-,\cdot)$ satisfies the law $s(t-u)=(st)-(su)$ for all $s,t,u\in A$.  All other laws remain satisfied since no new elements have been introduced.
\end{proof}

It is now trivial to add $\cap$ to the signature of both minus-algebras and minus-semigroups, using finitely many more equations as in Subsection \ref{cap}.  

\subsection{{\em Minus} and override}  \label{sub:minusover}

We next turn to the signature considered in \cite{minusover}, in which update can be defined as a term.  A finite equational axiomatisation was given there that was shown complete for equations.  The associated class of algebras was subsequently shown in \cite{CLS} to be equivalent to certain types of skew Boolean algebras (SBAs) previously considered by Leech in \cite{LeechSBA}, and because of this the axiomatisation was established as complete.  We here present a third axiomatisation, but go further by introducing function composition to the signature as well.

We say an algebra $(A,-,\sqcup)$ is a {\em minus-algebra with override} if it is a minus-algebra that additionally satisfies the following laws:
\bi
\item $(x\sqcup y)-x=y-x$;
\item $x\circ(x\sqcup y)=x$.
\ei

It is straightforward to check that override in $P(X,Y)$ satisfies these two laws.  

We say an algebra $(A,-,\sqcup,\cdot)$ is a {\em minus-semigroup with override} if:
\ben
\item $(A,-,\cdot)$ is a minus-semigroup; and
\item $(A,-,\sqcup)$ is a minus-algebra with override.
\een

Again, it is routine to verify that $(PT(X),-,\sqcup,\cdot)$ is a minus-semigroup with override, whence so is any subalgebra.

We have our usual result, which follows from Proposition \ref{minusnice}.

\begin{pro} \label{minusovernice}
Every minus-algebra with override is a minus-semigroup with override if we define $st=0$ for all $s,t$.
\end{pro}

\subsection{{\em Minus} and update}  \label{subsec:minusup}

We say $(A,-,\diamond)$ is a {\em minus-algebra with update} if
\ben
\item $(A,-)$ is a  minus-algebra,
\item $(x\diamond y)-x=x-(x\diamond y)=0$,   
\item $((x\diamond y)-y)\circ x=(x\diamond y)-y$, and 
\item $x\circ y=y\circ (x\diamond y)$.  
\een

We say $(A,\cdot,-,\diamond)$ is a {\em minus-semigroup with update} if
\ben
\item $(A,-,\diamond)$ is a minus-algebra with update, and
\item  $(A,\cdot,-)$ is a minus-semigroup.
\een

Again, the usual functional models satisfy these various laws, establishing soundness, and again from Proposition \ref{minusnice} we obtain the following.

\begin{pro} \label{minusupdatenice}
Every minus-algebra with update is a minus-semigroup with update if we define $st=0$ for all $s,t$.
\end{pro}

\subsection{Intersection and override}  \label{sub:intover}

In previous cases, we could simply add ``with intersection" to each set of axioms by just adding the relevant axioms from Subsection \ref{cap} to an existing axiom set.  In this signature and the one after, that is no longer possible: using our methods, we  cannot obtain axioms for domain restriction together with either just override or just update without first assuming intersection is present as well.

We say the algebra $(A,\circ,\cap,\sqcup)$ is a {\em right normal band with intersection and override} if the following laws are satisfied:
\ben
\item $(A,\circ,\cap)$ is a right normal band with intersection;
\item $s\circ (s\sqcup t)=s$;
\item $((s\sqcup t)\cap t)\sqcup s=s\sqcup t$;
\item $(s\sqcup t)\circ u=(s\circ u)\sqcup (t\circ u)$.
\een

As usual, it is a routine process to verify that $P(X,Y)$ equipped with domain restriction, intersection and override is a right normal band with intersection and override (and hence so are all its subalgebras).  

We add composition to the signature.  We say the algebra $(A,\circ,\cap,\sqcup,\cdot)$ is a {\em 1-stack with intersection and override} if the following laws are satisfied:
\ben
\item $(A,\circ,\cap)$ is a right normal band with intersection and override;
\item $(A,\cdot,\circ,\cap)$ is a 1-stack with intersection;
\item for all $s,t,u\in A$, $s(t\sqcup u)=st\sqcup su$.
\een

\begin{pro}  \label{overintnicembed}
Every right normal band with intersection and override embeds in a reduct of a 1-stack with intersection and override.
\end{pro}
\begin{proof} 
If $(A,\circ,\cap,\sqcup)$ is a right normal band with intersection and override, adjoin a zero element as in Proposition \ref{capembed}, also assuming that $s\sqcup 0=s$ and $0\sqcup s=s$ for all $s\in A\cup\{0\}$.  It follows from that result that $(A,\cdot,\circ,\cap)$ is a 1-stack with intersection, and routine case analysis shows that all three new laws for right normal bands with intersection and override as well as the third law above for 1-stacks with intersection and override are satisfied, so the result is a 1-stack with intersection and override into which $(A,\circ,\cap,\sqcup)$ embeds as a subreduct. 
\end{proof}

\subsection{Intersection and update}  \label{upint}

We say the algebra $(A,\circ,\cap,\diamond)$ is a {\em right normal band with intersection and update} if  and the following laws are satisfied: 
\ben
\item $(A,\circ,\cap)$ is a right normal band with intersection;
\item $s=(s\diamond t)\circ s$;
\item $s\diamond t=s\diamond(s\diamond t)$;
\item $s\circ t=t\circ(s\diamond t)$;
\item $(x\cap (x\diamond y))\circ a=(x\cap (x\diamond y))\circ b$ and $y\circ a=y\circ b$; imply $x\circ a=x\circ b$.
\een

As usual, there is a version involving composition as well.  We say  the algebra $(A,\cdot,\circ,\cap,\diamond)$ is a {\em 1-stack with intersection and update} if 
\ben
\item  $(A,\circ,\cap,\diamond)$ is a right normal band with intersection and update,
\item  $(A,\cdot,\circ,\cap)$ is a 1-stack with intersection, and
\item $s(t\diamond u)=st\diamond su$ for all $s,t,u\in A$.
\een

And as usual, we have an embedding result.

\begin{pro}  \label{updatenicembed}
Every right normal band with intersection and update embeds in a reduct of a 1-stack with intersection and update.
\end{pro}
\begin{proof}
If $(A,\circ,\cap,\diamond)$ is a  right normal band with intersection and update, adjoin a zero element as in Proposition \ref{capembed}, also assuming that $s\diamond 0=s$ and $0\diamond s=0$ for all $s\in A\cup\{0\}$.  Then argue as for Proposition \ref{overintnicembed}.
\end{proof}

\subsection{Override only}  \label{overonly}

The signature $(\circ,\sqcup)$ does not lend itself to axiomatisation using the methods presented here.  For completeness, we here list axioms for this case, based on those presented in \cite{LeechNSL} for certain types of skew lattices, which we now define.

The algebra  $(A,\circ,\sqcup)$ is a {\em right handed strongly distributive skew lattice} if the following laws are satisfied:
\ben
\item $(A,\circ)$ is a right normal band (indeed it suffices that it be a band satisfying the law $x\circ y\circ x=y\circ x$);
\item $(A,\sqcup)$ is a band (idempotent semigroup);
\item the absorption laws $x\circ (x\sqcup y)=x=(y\sqcup x)\circ x$ and $x\sqcup (x\circ y)=x=(y\circ x)\sqcup x=x$ hold;
\item the distributive laws $s\circ (t\sqcup u)=(s\circ t)\sqcup (s\circ u)$ and $(s\sqcup t)\circ u=(s\circ u)\sqcup (t\circ u)$ hold.
\een

\section{Quasiorders and filters} \label{sec:filters}

\subsection{Right normal bands in general}

For the rest of this section, let $(A,\circ)$ be a fixed right normal band.
We may define two quasiorders on $(A,\circ)$ as follows:
\bi
\item $f\lesssim g$ if and only if $g\circ f=f$; in $P(X,Y)$ this asserts that the domain of $f$ is contained in that of $g$, so $\lesssim$ is the {\em first projection quasiorder}, which was axiomatised within function semigroups by Schein in \cite{BStrans};
\item $f\leq g$ if and only if $f=f\circ g$; in $P(X,Y)$ this asserts that $f\subseteq g$, or $f$ is a domain restriction of $g$, and we call it the {\em natural order},  and we apply the same terminology to right normal bands in general.  The natural order was also axiomatised within function semigroups by Schein in \cite{BStrans}.
\ei

The next facts follow from Vagner's representation theorem as in \cite{vagner}, but direct proofs are straightforward.

\begin{pro}
On $(A,\circ)$, $\lesssim$ is a quasiorder and $\leq$ is a partial order, and $a\leq b$ implies $a\lesssim b$ for all $a,b\in A$.  Moreover the equivalence relation $\sim$ determined by $\lesssim$ is a congruence, and $(A/{\sim},\circ)$ is a semilattice.  
\end{pro}
\begin{proof}
Both relations are quasiorders because $(A,\circ)$ is a band.  If $a\leq b$ and $b\leq a$, then $a=a\circ b$ and $b=b\circ a$, so $$a=a\circ b=a\circ(b\circ a)=(a\circ b)\circ a=(b\circ a)\circ a=b\circ a=b.$$  So $\leq$ is a partial order.  If $a\leq b$ then $a=a\circ b$ and so $$b\circ a=b\circ(a\circ b)=(b\circ a)\circ b=(a\circ b)\circ b=a\circ(b\circ b)=a\circ b=a,$$ so $a\lesssim b$.  

Suppose $a\sim b$ and $c\sim d $.  Then $a\circ b=b$, $b\circ a=a$, $c\circ d=d$, $d\circ c=c$.  So $$(a\circ c)\circ (b\circ d)=a\circ b\circ c\circ d=b\circ d,$$ 
so $b\circ d\lesssim a\circ c$.  By symmetry, $a\circ c\lesssim b\circ d$, so $a\circ c\sim b\circ d$.

Now $(b\circ c)\circ (c\circ b)=b\circ c\circ b=c\circ b\circ b=c\circ b$ so $b\circ c\lesssim c\circ b$; so by symmetry $b\circ c\sim c\circ b$.  So $(A/{\sim},\circ)$ is a commutative band, hence a semilattice.
\end{proof}

Let $(A,\circ)$ be a fixed right normal band throughout the remainder of this section.  Let $F\subseteq A$ be a filter, meaning a non-empty subset of $A$ such that for all $a,b\in F$, $a\circ b\in F$, and if $a\in F$ and $a\lesssim b$, then $b\in F$ (so $F$ is an {\em up-set} under $\lesssim$).  (So in other words, $F/{\sim}$ is a filter in $A/{\sim}$ in the usual sense for semilattices.)


The next observation is useful in what follows.

\begin{lem}  \label{filtextend}
Let $F$ be a proper filter of $A$ and $a\in A\backslash F$.  Then  
$$F_a=\{b\in A\mid f\circ a\lesssim b\mbox{ for some }f\in F\}$$
is a filter of $A$ containing $a$ and $F$.
\end{lem}
\begin{proof}
If $b_1,b_2\in F_a$ then there are $f_1,f_2\in F$ such that $f_1\circ a\lesssim b_1$ and $f_2\circ a\lesssim b_2$, so letting $f=f_1\circ f_2\in F$, we have
$$b_1\circ b_2\circ f\circ a=b_2\circ f_2\circ b_1\circ f_1\circ a=b_2\circ f_2\circ f_1\circ a=f_1\circ b_2\circ f_2\circ a=f_1\circ f_2\circ a=f\circ a,$$ so $f\circ a \lesssim b_1\circ b_2$, so $b_1\circ b_2\in F_a$.   Obviously $F_a$ is an up-set, hence it is a  filter.  Since $f\circ a\lesssim a$ for any $f\in F$, we have $a\in F_a$, and since also $f\circ a\lesssim f$ for all $f\in F$, we have $F\subseteq F_a$. 
\end{proof}

Let $(A,\circ)$ be a right normal band.  Suppose $a,b\in A$ are such that $a\not\leq b$.  We say the filter $F$ of $A$ is {\em $(a,b)$-separating} if (i) $a\in F$, and (ii) there is no $e\in F$ for which $e\circ a=e\circ b$.  

For $a\in A$, denote by $a^{\uparrow}$ the principal filter of $A$ generated by $a$, so $a^{\uparrow}=\{e\in A\mid a\lesssim e\}$.  Clearly this is an  
up-set under $\lesssim$, and if $e,f\in a^{\uparrow}$, then $e\circ a=f\circ a=a$, so $(e\circ f)\circ a=e\circ (f\circ a)=e\circ a=a$, and so $e\circ f\in F$.

\begin{lem}  \label{sep}
If $a,b\in S$ with $a\not\leq b$, then $F=a^{\uparrow}$ is $(a,b)$-separating.
\end{lem}
\begin{proof}
Evidently $a\in F$.  Suppose there is $e\in F$ for which $e\circ a=e\circ b$.  Since $a\lesssim e$, $a=e\circ a$.  So $a\circ b=e\circ a\circ b=a\circ e\circ b=a\circ e\circ a=e\circ a\circ a=e\circ a=a$, and so $a\leq b$, a contradiction.  So $F$ is $(a,b)$-separating.
\end{proof}

Let $F$ be a proper filter of $(A,\circ)$.  Define a binary relation $\epsilon_F$ on $A$ by setting
$$\mathrel{\epsilon_F} = \{ (a,b)\in A\times A\mid e\circ a=e\circ b\mbox{ for some }e\in F\}.$$

\begin{pro}  \label{congeps}
If $F$ is a proper filter of $A$, then $\epsilon_F$ is a congruence on $(A,\circ)$, and both $F$ and $\overline{F}$ (the complement of $F$ in $A$) are unions of $\epsilon_F$-classes.
\end{pro}
\begin{proof}
Reflexivity and symmetry are immediate.  If $a,b,c\in F$, and $e\circ a=e\circ b$ and $f\circ b=f\circ c$ for some $e,f\in F$, then letting $g=e\circ f\in F$, we have
$$g\circ a=e\circ f\circ a=f\circ e\circ a=f\circ e\circ b=e\circ f\circ b=e\circ f\circ c=g\circ c.$$  

Suppose $(a,b),(c,d)\in \epsilon_F$.  Then there must be $e,f\in F$ such that $e\circ a=e\circ b$ and $f\circ c=f\circ d$.  Again, letting $g=e\circ f\in F$, we have that 
$$(e\circ f)\circ (a\circ c)=f\circ e\circ a\circ c=f\circ e\circ b\circ c
=f\circ b\circ e\circ c=f\circ b\circ e\circ d=(e\circ f)\circ (b\circ d).$$
So $(a\circ c,b\circ d)\in \epsilon_F$.

If $x\in \overline{F}$ and $(x,y)\in \epsilon_F$, then $e\circ x=e\circ y$ for some $e\in F$, so if $y\in F$, then $e\circ y$ and hence $e\circ x$ are in $F$, so $x\in F$, a contradiction, so $y\in \overline{F}$.  Hence $\overline{F}$ is a union of $\epsilon_F$-classes, and so $F$ must be as well.
\end{proof}

Let $X$ be a set of filters of $A$ with the property that for every $a,b\in A$ with $a\not\leq b$, there exists $F\in X$ that is $(a,b)$-separating.  We call such $X$ {\em separating}.  Lemma \ref{sep} now gives the following.

\begin{pro}
The set of all filters of $A$ is separating.
\end{pro}

Suppose $a,b\in A$, with $a\not\leq b$.  We say the filter $F$ of $(A,\circ)$ is {\em maximally $(a,b)$-separating} if it is maximal with respect to the property of being $(a,b)$-separating.   Because $F=a^{\uparrow}$ is $(a,b)$-separating by Lemma \ref{sep}, we readily obtain the following from Zorn's Lemma.

\begin{lem} \label{maxsep}
For any $a,b\in S$ for which $a\not\leq b$, there is a maximally $(a,b)$-separating filter in $A$.
\end{lem}

It is immediate that the set ${\mc F}$ of all filters of $A$ that is maximally $(a,b)$-separating, ranging over all $a,b\in A$ for which $a\not\leq b$, is separating.

Another useful fact is the following.

\begin{lem}  \label{maxeq}
If $(A,\circ)$ is a right normal band with $a\not\leq b$ in $A$, and $F$ a maximally $(a,b)$-separating filter in $A$ with $y\in A\backslash F$, then there exists $f\in F$ for which $(f\circ y)\circ a = (f\circ y)\circ b$.
\end{lem}
\begin{proof}
Now $F_{y}=\{g\in A\mid y\circ f\lesssim g\mbox{ for some } f\in F\}$ as in Lemma \ref{filtextend} properly contains $y$ and $F$, hence is not $(a,b)$-separating by maximality of $F$, and so there exists $e\in F_{y}$, that is $f\circ y\lesssim e$ for some $f\in F$, such that $e\circ a=e\circ b$.  Hence also $(f\circ y)\circ a = (f\circ y)\circ b$ (because if $e\lesssim f$ and $f\circ s=f\circ t$ then $e\circ s=(e\circ g)\circ s=e\circ(g \circ s)=e\circ(g\circ t)=(e\circ g)\circ t=e\circ t$).
\end{proof}

\subsection{Right normal bands with intersection}

Throughout this subsection, let $(A,\circ,\cap)$ be a fixed right normal band with intersection.  In this case, the natural order on $(A,\circ)$ coincides with the partial order determined by the meet-semilattice $(A,\cap)$.

\begin{pro}  \label{leqcap}
For $x,y\in A$, $a\leq b$ if and only if $a=a\cap b$.
\end{pro}
\begin{proof}
Suppose $a\leq b$, so that $a=a\circ b$.  Then 
$$a\cap b=(a\circ b)\cap (b\circ b)=(a\circ b)\cap b=a\circ(b\cap b)=a\circ b=a.$$
Conversely, suppose $a=a\cap b$.  Then $a\circ b=(a\cap b)\circ b=a\cap b=a$, so $a\leq b$.
\end{proof}

\begin{lem} \label{star}
For any filter $F$ of $A$ and $x,y\in A$, it is the case that $x,y\in F$ and $(x,y)\in \epsilon_F$ if and only if $x\cap y\in F$.
\end{lem}
\begin{proof}
Suppose $x\cap y\in F$.  Then $x,y\in F$ since $x\cap y \leq x,y$ by Proposition \ref{leqcap}, and so $x\cap y\lesssim x,y$.  But also, $(x\cap y)\circ x=x\cap y=(x\cap y)\circ y$, so $(x,y)\in \epsilon_F$.

Conversely, suppose  $x,y\in F$ and $(x,y)\in \epsilon_F$, so there exists $e\in A$ for which $e\circ x=e\circ y$.  So $e\circ (x\cap y) = (e\circ x)\cap (e\circ y) = (e\circ x)\cap (e\circ x)=e\circ x$, and so $F\ni e\circ x\lesssim x\cap y$, so $x\cap y\in F$.
\end{proof}

\subsection{Minus-algebras}  \label{sub:minus}

We begin with a useful result.

\begin{lem}\label{altqi}
In any minus-algebra, the following law holds:
$$(z-x)\circ s=(z-x)\circ t\And x\circ s=x\circ t\Rightarrow z\circ s=z\circ t.$$
\end{lem}
\begin{proof}
Assume that $(z-x)\circ s=(z-x)\circ t$ and $x\circ s = x\circ t$ hold; we must show that $z\circ s= z\circ t$.  Now $x\circ s = x\circ t$ implies that $z\circ x\circ s = z\circ x\circ t$, which in turn gives 
\begin{equation}
x\circ (z\circ s) = x\circ (z\circ t),\label{eq:step1}
\end{equation} 
using the laws of domain restriction.  Next $(z-x)\circ s = (z\circ s)-x$ and $(z-x)\circ t=(z\circ t)-x$ so that 
\begin{equation}
(z\circ s)-x= (z\circ t)-x.\label{eq:step2}
\end{equation}   
Applying the final law for minus-algebras to \eqref{eq:step1} and \eqref{eq:step2} gives $z\circ s=z\circ t$ as required. 
\end{proof}

For reasons that will become more plain later, let us say that a filter $F$ of the minus-algebra $(A,-)$ (viewed as a right normal band) is {\em prime} if, whenever $a\in F$, for all $b\in A$, either $b\in F$ or $a-b\in F$.

\begin{lem}  \label{prime-}
Let $F$ be maximally $(a,b)$-separating, for some $a,b\in A$ with $a\not\leq b$.  Then $F$ is prime.
\end{lem}
\begin{proof} 
First suppose that $x\in F$ and $y\not\in F$, and for a contradiction that $x-y\not\in F$.  Then by Lemma~\ref{maxeq}, there exists $f\in F$ for which $(f\circ y)\circ a = (f\circ y)\circ b$.  Similarly, there exists $h\in F$ for which $(h\circ (x-y))\circ a = (h\circ (x-y))\circ b$.  So letting $k=f\circ h\in F$, we have $y\circ(k\circ a)=y\circ(k\circ b)$ and $(x-y)\circ(k\circ a)=(x-y)\circ(k\circ b)$,
so $(x\circ k)\circ a=x\circ(k\circ a)=x\circ(k\circ b)=(x\circ k)\circ b$ by Lemma \ref{altqi}.  But $x\circ k\in F$, so this contradicts the $(a,b)$-separating property of $F$.  Hence in fact $x-y\in F$. 
\end{proof}

\begin{cor}
The set of all prime filters of $(A,-)$ is separating.
\end{cor}

\subsection{Minus-algebras with override}

Of interest is the following alternative way of looking at prime filters in minus-algebras with override, making clear that the termionology is natural.

\begin{pro}  \label{prime}
Let  $(A,-,\sqcup)$ be a minus-algebra with override.  A filter $F$ of the minus-algebra $(A,-)$ is prime if and only if it satisfies the following:
for all $a,b\in A$, if $a\sqcup b\in F$ then $a\in F$ or $b\in F$.
\end{pro}
\begin{proof}
Suppose $F$ is prime, and that $a\sqcup b\in F$.  Suppose $b\not\in F$.  Then by primeness, $F\ni (a\sqcup b)-b\lesssim a$ (as follows from the functional interpretation afforded by Proposition \ref{overmain}), and so $a\in F$.

Conversely, suppose that for all $a,b\in A$, if $a\sqcup b\in F$ then $a\in F$ or $b\in F$.  Suppose $a\in F$, and $b\in A$.  Then because $a\lesssim (a-b)\sqcup b$ (again from the functional interpretation), we have $(a-b)\sqcup b\in F$ and so $a-b\in F$ or $b\in F$.  So $F$ is prime.
\end{proof}

\subsection{Right normal bands with intersection and update}

Again we need a relevant notion of ``prime" for filters in this case.  Suppose $(A,\circ,\cap,\diamond)$ is a right normal band with intersection and update.  This time we say the filter $F$ of $(A,\circ)$ is {\em weakly prime} if whenever $a\in F$, for all $b\in A$, either $b\in F$ or $a\cap(a\diamond b)\in F$.

\begin{lem}  \label{primediamond}
Suppose $(A,\circ,\cap,\diamond)$ is a right normal band with intersection and update.  Pick $a,b\in A$ with $a\not\leq b$, and let $F$ be maximally $(a,b)$-separating in $(A,\circ)$.  Then $F$ is weakly prime.
\end{lem}
\begin{proof} 
First suppose that $x\in F$ and $y\not\in F$, and for a contradiction that $x\cap (x\diamond y)\not\in F$.  Then by Lemma \ref{maxeq}, there is $f\in F$ for which $(f\circ y)\circ a = (f\circ y)\circ b$.  Similarly, there exists $h\in F$ for which $(h\circ (x\cap (x\diamond y)))\circ a = (h\circ (x\cap (x\diamond y)))\circ b$.  So letting $k=f\circ h\in F$, we have $y\circ(k\circ a)=y\circ(k\circ b)$ and $(x\cap (x\diamond y))\circ(k\circ a)=(x\cap (x\diamond y))\circ(k\circ b)$, 
so $(x\circ k)\circ a=x\circ(k\circ a)=x\circ(k\circ b)=(x\circ k)\circ b$ by the implication law for right normal bands with intersection and update.  But $x\circ k\in F$, so this contradicts the $(a,b)$-separating property of $F$.  Hence indeed $x\cap (x\diamond y)\in F$. 
\end{proof}

\begin{cor}
The set of all weakly prime filters of a right normal band with intersection and update is separating.
\end{cor}

\subsection{Right normal bands with intersection and override}  \label{intover}

Let $(A,\circ,\cap,\sqcup)$ be a fixed right normal band with intersection and override.  Following Proposition \ref{prime}, we say a filter $F$ of $(A,\circ,\cap,\sqcup)$ is {\em prime} if and only if it satisfies the following:
for all $a,b\in A$, if $a\sqcup b\in F$ then $a\in F$ or $b\in F$. 

\begin{lem}  \label{primecap}
For $a,b\in A$ with $a\not\leq b$, every maximally $(a,b)$-separating filter of $A$ is prime.
\end{lem}
\begin{proof}
Let $F$ be maximally $(a,b)$-separating.  Suppose $s\sqcup t\in F$ and for a contradiction that $s\not\in F$ and $t\not\in F$.  Then the filter $F_s=\{g\in A\mid s\circ f\lesssim g\mbox{ for some }f\in F\}$ as in Lemma \ref{filtextend} contains $F$ and $s\not\in F$ (and $a$), so by maximality of $F$, there exists $g_1\in F_s$ for which $g_1\circ a=g_1\circ b$.  So arguing as in the proof of Lemma \ref{prime-}, $s\circ f_1\circ a=s\circ f_1\circ b$ for some $f_1\in F$.  Similarly, there is $f_2\in F$ for which $t\circ f_2\circ a=t\circ f_2\circ b$.  So letting $f=f_1\circ f_2\in F$, we have
$$s\circ f\circ a=s\circ f\circ b,\ t\circ f\circ a=t\circ f\circ b.$$  
So 
\bea
(s\sqcup t)\circ f\circ a&=&(s\circ f\circ a)\sqcup (t\circ f\circ a)\\
&&\mbox{ by the third law for 1-stacks with intersection and override}\\
&=&(s\circ f\circ b)\sqcup (t\circ f\circ b)\\
&=&(s\sqcup t)\circ f\circ b,
\eea
so $g\circ a=g\circ b$ where $g=(s\sqcup t)\circ f\in F$, contradicting the fact that $F$ is $(a,b)$-separating.  So $F$ is prime.
\end{proof}

\begin{cor}
The set of all prime filters of $(A,\circ,\cap,\sqcup)$ is separating.
\end{cor}

Of course every right normal band with intersection and override is a right normal band with intersection and update, which we define via $a\diamond b=a\circ(a\sqcup b)$ as usual.  We do not know if the weakly prime filters in a right normal band with intersection and override viewed in this way as a right normal band with intersection and update are nothing but its prime filters.  However in general we have the following, which justifies use of the ``weakly" epithet.

\begin{pro}
Every prime filter of the right normal band with intersection and override $(A,\circ,\cap,\sqcup)$ is weakly prime.
\end{pro}
\begin{proof}
Suppose $F$ is a prime filter, with $x\in F$ and $y\not\in F$.  Now $x=(x\cap(x\diamond y))\sqcup y\circ x$ as follows from the functional interpretation (and so must follow from the laws), so as $y\circ x\lesssim y$ we have $y\circ x\not\in F$ also, and then by primeness, we obtain $x\cap(x\diamond y)\in F$.
\end{proof}

\section{The representations}  \label{sec:rep}

The same basic construction is used throughout in what follows.  The main advantage of this is that it is relatively straightforward to enrich the signature of domain restriction by including other operations defined on partial functions in $P(X,Y)$, such as minus, intersection, override, but especially function composition when $Y=X$, since the construction will represent these operations correctly also for some choice of separating filters ${\mc F}$.  Indeed our approach is to represent algebras with signature containing an operation modelling composition first, and use this to obtain the analogous results for those without composition.

Let $S$ be a 1-stack, and let $F$ be a filter of $A$, viewed as a right normal band.  Let $A^1$ be $A$ with adjoined identity element $1$, so that $1a=a1=a$ for all $a\in A$,  and extend $\epsilon_F$ to $A^1$ by putting~$\{1\}$ in a class by itself, and then let $A_F=(F\cup\{1\})/\epsilon_F$ (recalling that $F$ is a union of $\epsilon_F$-classes by Proposition \ref{congeps}, so the quotient is defined).  For $x\in A^1$, denote by $\overline{x}^F$ the $\epsilon_F$-class containing $x$.  For any $a\in A$, define $\phi^F_a:A_F\rightarrow A_F$,
by setting, for $\overline{x}^F\in A$,
\[
\phi^F_a(\overline{x}^F):=\begin{cases} \overline{xa}^F&\mbox{ providing }xa\in F,\\
\mbox{undefined}&\mbox{ otherwise.}
\end{cases}
\]
So $\phi^F_a$ is a partial function in $PT(A_F)$, for each $a\in A$.  

Now for ${\mathcal F}$ any separating set of filters of $A$, view the $A_F$ ($F\in {\mathcal F}$) as mutually disjoint, let $X=\bigcup \{A_F\mid F\in {\mathcal F}\}$ be their disjoint union, and let $\phi_a=\bigcup \{\phi^F_a\mid F\in {\mathcal F}\}$ be the disjoint union of the $\phi^F_a$ as $F$ ranges across ${\mathcal F}$; so $\phi_a\in PT(X)$ for all $a\in A$.  

\begin{pro}  \label{main}
With the above definitions, the mapping $\Phi: A\rightarrow PT(X)$ given by $\Phi(a)=\phi_a$ for all $a\in S$ is a 1-stack embedding, mapping any zero element to the empty function.
\end{pro}
\begin{proof}
We first show that if $x,y\in A$ and $(x,y)\in \epsilon_F$, then for all $a\in A$, $(xa,ya)\in \epsilon_F$ (noting that if $x=1$, this is true also).  If $e\circ x=e\circ y$ then $e\circ (xa)=e\circ x\circ (xa) =e\circ y\circ (xa)=y\circ e\circ (xa)=y\circ ((e\circ x)a)=y\circ ((e\circ y)a)=y\circ e\circ (ya)=e\circ y\circ (ya)=e\circ (ya)$, so $(xa,ya)\in \epsilon_F$. 

Next we show that $\overline{F}$, the complement of $F$ in $A$, is a right ideal if non-empty.  For $s\in \overline{F}$, because $s\circ (st)=st$ for all $t\in A$, it follows that $st\lesssim s$, so if $st\in F$ then $s\in F$, a contradiction, so $st\in \overline{F}$ also.   We note also that $\overline{F}$ is a union of $\epsilon_F$-classes, from Proposition \ref{congeps}.

It now follows from the theory of determinative pairs (due to Boris Schein and since used in many settings) that the mapping $A\rightarrow PT(A_F)$ given by $a\mapsto \psi^F_a$ is a semigroup homomorphism, and indeed that $\Phi$ is a homomorphism as well, since the computation of the operations takes place independently on each $A_F$-patch of $X$, for those elements of $PT(X)$ in the range of $\Phi$.  

If $s\not\leq t$, there exists $F\in {\mathcal F}$ that is $(s,t)$-separating.  So $\psi^F_s(1)$ is defined, and even if $\psi^F_t(1)$ is defined as well, so that $t\in F$, it would be the case that $(s,t)\not\in \epsilon_F$ and so $\psi^F_s(1)\neq\psi^F_t(1)$.  So $\psi^F(s)\not\subseteq \psi^F(t)$, and so $\Phi(s)\not\subseteq \Phi(t)$.  This shows that $\Phi$ is an embedding.

Next we show $\Phi$ respects domain restriction.  Now for $s,t\in A$, $\Phi(s\circ t)$ is defined at  $\overline{x}^F\in A_F$, if and only if $\psi^F(s\circ t)$ is defined at $\overline{x}^F$, if and only if $(xs)\circ (xt)=x(s\circ t)\in F$, that is, $xs,xt\in F$ (since $xs\circ xt \lesssim xs,xt$), or $\psi^F(s),\psi^F(t)$ are both defined at $\overline{x}^F$, so $\Phi(s),\Phi(t)$ are both defined at $\overline{x}^F$.  For such $\overline{x}^F$, $x(s\circ t)=xs\circ(xs\circ xt)=xs\circ xt$, so since $xs\in F$, we have $(x(s\circ t),xt)\in \epsilon_F$, and so $\psi^F_{s\circ t}(\ol{x})=\psi^F_t(\ol{x})$, and so $\Phi(s\circ t)(\ol{x}^F)=\Phi(t)(\ol{x}^F)$.  So by definition,  $\Phi(s\circ t)=\Phi(s)\circ \Phi(t)$. 

If $(S,\circ)$ has a zero element $0$, then it is a smallest element under $\lesssim$, so $0\not\in F$ for any filter $F\in {\mathcal F}$, and so $\Phi(0)=\varnothing$.
\end{proof}

\begin{cor}
The algebras of partial functions of signature $(\cdot,\circ)$ are axiomatised as the class of 1-stacks, and the algebras of partial functions of signature $(\circ)$ are axiomatised as the class of right normal bands.
\end{cor}

Each of these facts is well-known; see \cite{schein} and \cite{vagner}.  The value of the construction just used is that it is sufficiently versatile to admit representation theorems for richer signatures as well, including the one considered in \cite{minusover}.  We now present a single result that addresses every one of the cases so far considered.  

\begin{thm} \label{main2}
Let $(A,\cdot,\circ,\ldots)$ be a 1-stack that is enriched in one of the ways described in Section~\ref{sec:axioms}, and recall the 1-stack embedding $\Phi: A\rightarrow PT(X)$ as in Proposition \ref{main}.
\ben
\item  \label{intmain} If $(A,\cdot,\circ,\cap)$ is a 1-stack with intersection, then for any choice of separating ${\mc F}$, $\Phi$ respects intersection.
\item  \label{minusmain} If $(A,\cdot,-)$ is a minus-semigroup, with ${\mc F}$ the set of all prime filters of $(A,-)$, then $\Phi$ respects minus.
\item  \label{overmain} If $(A,\cdot,-,\sqcup)$ is a minus-semigroup with override, with ${\mc F}$ the set of all prime filters of $(A,-)$, then $\Phi$ respects minus and override.  
\item  \label{updatemain} If $(A,\cdot,-,\diamond)$ is a minus-semigroup with update, with ${\mc F}$ the set of all prime filters of $(A,-)$, then $\Phi$ respects minus and update.
\item \label{minusmore} If $(A,\cdot,-,\ldots)$ is any of the last three cases involving minus, and has intersection added to it by assuming the 1-stack with intersection laws, then $\Phi$ respects intersection.
\item \label{csqmain}  If $(A,\circ,\cap,\sqcup,\cdot)$ is a 1-stack with intersection and override, and ${\mc F}$ is the set of all prime filters of $(A,\circ,\cap,\sqcup)$, then $\Phi$ respects $\cap$ and $\sqcup$.  
\item  \label{diamondmain} If $(A,\cdot,\circ,\cap,\diamond)$ is a 1-stack with intersection and update, and ${\mc F}$ is the set of all weakly prime filters of $(A,\diamond,\cap)$, then $\Phi$ respects $\cap$ and $\diamond$.
\een
\end{thm}

\begin{proof}
For (\ref{intmain}), note that $\ol{x}^F$ is in the domain of $\Phi(s)\cap \Phi(t)$ if and only if it is in the domain of $\psi^F_s\cap \psi^F_t$, that is, $xs\in F$, $xt\in F$, and $(xs,xt)\in \epsilon_F$, which by the previous lemma is equivalent to saying that $x(s\cap t)=xs\cap xt\in F$ (using 3 in the definition if $x\in A$, and the fact that if $x=1$, this is trivially true), which is to say that $\ol{x}^F$ is in the domain of $\psi^F_{s\cap t}=\Phi(s\cap t)$.  So the domains of $\Phi(s\cap t)$ and $\Phi(s)\cap \Phi(t)$ coincide.  But for such $\ol{x}^F$ in this common domain, $(xs)\cap(xs\cap xt)=xs\cap xt\in F$, so $(xs,xs\cap xt)\in \epsilon_F$ by Lemma \ref{star}, and so 
$$\Phi(s\cap t)(\ol{x}^F)=\psi^F_{s\cap t}(\ol{x}^F)=\ol{x(s\cap t)}^F=\ol{(xs\cap xt)}^F=\ol{xs}^F=\psi^F_s(\ol{x}^F)=\Phi(s)(\ol{x}^F),$$ and similarly $\Phi(s\cap t)(\ol{x}^F)=\Phi(t)(\ol{x}^F)$.  So $\Phi(s\cap t)=\Phi(s)\cap \Phi(t)$.

For (\ref{minusmain}), $\ol{x}^F\in X$ is in the domain of $\Phi(s)-\Phi(t)$ if and only if $xs\in F$ and $xt\not\in F$, which implies that $x(s-t)=xs-xt\in F$ by primeness, and the minus-semigroup law if $x\in A$, with it being trivial if $x=1$; indeed it is equivalent to this since if $xs-xt\in F$ then $xs\in F$ since $xs-xt\lesssim xs$ (as $xs\circ(xs-xt)=xs-xt$ by the second additional law for minus-semigroups), and if $xt\in F$ then by the third law for minus-algebras, $(xs-xt)\circ (xt)=0\in F$, a contradiction.  But $xs-xt\in F$ simply states that $\Phi(s-t)$ is defined at $\ol{x}^F$.  When this happens, $(xs-xt)\circ (xs-xt)=xs-xt=(xs-xt)\circ xs$, so $(xs,xs-xt)\in \epsilon_F$, and so $\Phi(s-t)=\Phi(s)$.  So by definition, $\Phi(s-t)=\Phi(s)-\Phi(t)$.

For (\ref{overmain}), it suffices to show that a minus-algebra $(A,-)$ of functions equipped with an operation $\sqcup$ satisfying the additional two laws for minus-algebras with override that involve $\sqcup$ must have $\sqcup$ equal to override.  (In \cite{jacstoverup}, the general idea behind this is called  {\em abstract definability}, in this case of $\sqcup$ from minus and domain restriction.)  The second law says that restricting the partial function $x\sqcup y$ to the domain of $x$ yields $x$, while the first says that if this part of $x\sqcup y$ is removed, the remainder is the same as what one obtains by restricting $y$ to where $x$ is undefined.  So in summary, $x\sqcup y$ is the union of $x$ with the restriction of $y$ to where $x$ is undefined, which is nothing but their preferential union, or override of $y$ by $x$.  The proof of (\ref{updatemain}) is very similar.

For each of the signatures involving minus as in (\ref{minusmore}), if the 1-stack with intersection axioms are added, then $\Phi$ respects intersection by (\ref{intmain}) (already shown).

Now we turn to (\ref{csqmain}).  Let $s,t\in A$.  We must show that $\Phi(s\sqcup t)=\Phi(s)\sqcup \Phi(t)$ as partial functions.  First we show their domains are equal.  Now $\ol{x}^F\in \dom(\Phi(s\sqcup t))$ says that $x(s\sqcup t)\in F$.  But $x(s\sqcup t)=xs\sqcup xt$ by 3 in the definition of 1-stacks with intersection and override if $x\in A$, and trivially if $x=1$, so this is equivalent to saying that $xs\in F$ or $xt\in F$, which is equivalent to saying that $\ol{x}^F$ is in $\dom(\Phi(s))\cup \dom(\Phi(t))$.  So the domains are equal.  For $\ol{x}^F$ in this common domain, we consider two cases.  (i) If $\Phi(s)$ is defined at $\ol{x}$, so that $xs\in F$, then $xs\circ(xs\sqcup xt)=xs=xs\circ xs$ by the first law for right normal bands with intersection and override, so $\Phi(s\sqcup t)$ and $\Phi(s)$ agree at $\ol{x}$.  (ii) If $\Phi(s)$ is not defined at $\ol{x}$, so that $xs\not\in F$, then necessarily $xt\in F$, and also $F\ni xs\sqcup xt=((xs\sqcup xt)\cap xt)\sqcup xs$ by (3) in the definition of right normal bands with intersection and override, so by primeness of $F$, $(xs\sqcup xt)\cap xt\in F$, and so by Lemma \ref{star}, $x(s\sqcup t)=xs\sqcup xt\ \epsilon_F\ xt$, and so $\Phi(s\sqcup t)$ and $\Phi(t)$ agree at $\ol{x}^F$.  Overall then $\Phi(s\sqcup t)$ and $\Phi(s)\sqcup \Phi(t)$ agree at all $\ol{x}^F$ in their (equal) domains, hence are the same partial functions.

Finally, we consider (\ref{diamondmain}).  Pick $s,t\in A$.  Now $\ol{x}^F\in X$ is in the domain of $\Phi(s)\diamond\Phi(t)$ if and only if $\ol{x}^F$ is in the domain of $\Phi(s)$, that is, $xs\in F$, or equivalently (even if $x=1$ as for $\sqcup$ above), $x(s\diamond t)=xs\diamond xt\in F$ (since $xs\diamond xt\sim xs$), that is, $\ol{x}^F$ is in the domain of $\Phi(s\diamond t)$.  For such $\ol{x}^F$, we consider cases.  (i) If $\ol{x}^F$ is in the domain of $\Phi(t)$, then $xt\in F$ and so $xt\circ xs\in F$, so $(xt\circ xs)\circ xt=xs\circ xt=xt\circ (xs\diamond xt)=xt\circ (xs\circ (xs\diamond xt))=(xt\circ xs)\circ (xs\diamond xt)$, we have $(x(s\diamond t),xt)\in \epsilon_F$, and so $\Phi(s\diamond t)(\ol{x}^F)=\Phi(t)(\ol{x}^F)$.  (ii) If $\ol{x}^F$ is not in the domain of $\Phi(t)$, then $xt\not\in F$ and so by the weakly prime property, $xs\cap (xs\diamond xt)\in F$, and so 
because $xs\diamond xt=x(s\diamond t)$, we have $(xs,x(s\diamond t))\in \epsilon_F$ by Lemma \ref{star}, and so $\Phi(s\diamond t)=\Phi(s)$ at $\ol{x}^F$.   So $\Phi(s\diamond t)$ and $\Phi(s)\diamond\Phi(t)$ agree on their common domain and hence are equal.
\end{proof}

From this theorem we obtain the following.

\begin{thm}  \label{corstack}
The algebras of partial functions of each of the signatures containing domain restriction and composition listed below are axiomatised as the indicated class of enriched 1-stacks.
\ben
\item $(\cdot,\circ,\cap)$ $\leftrightarrow$ 1-stacks with intersection. 
\item $(\cdot,-)$ $\leftrightarrow$ minus-semigroups.  
\item $(\cdot,-,\sqcup)$  $\leftrightarrow$ minus-semigroups with override. 
\item  $(\cdot,-,\diamond)$ $\leftrightarrow$ minus-semigroups with update.  
\item Any of the last three with intersection added $\leftrightarrow$ the relevant axioms plus those for 1-stacks with intersection. 
\item $(\cdot,\circ,\cap,\sqcup)$ $\leftrightarrow$ 1-stacks with intersection and override. 
\item $(\cdot,\circ,\cap,\diamond)$ $\leftrightarrow$ 1-stacks with intersection and update.  
\een
\end{thm}

Then from this and using the relevant one of Propositions \ref{capembed}, \ref{minusnice}, \ref{minusovernice}, \ref{minusupdatenice},  \ref{overintnicembed} and
\ref{updatenicembed}, we obtain the following further consequence. 

\begin{cor} \label{coralgebra}
The algebras of partial functions of each of the signatures containing domain restriction but not composition listed below are axiomatised as the indicated class of enriched right normal bands.
\ben
\item $(\circ,\cap)$  $\leftrightarrow$ right normal bands with intersection.
\item $(-)$ $\leftrightarrow$ minus-algebras.
\item  $(-,\sqcup)$ $\leftrightarrow$ minus-algebras with override. 
\item $(-,\diamond)$ $\leftrightarrow$ minus-algebras with update. 
\item  Any of the last three with intersection added $\leftrightarrow$ the relevant axioms plus those for right normal bands with intersection. 
\item $(\circ,\cap,\sqcup)$ $\leftrightarrow$ right normal bands with intersection and override.
\item $(\circ,\cap,\diamond)$ $\leftrightarrow$ right normal bands with intersection and update.
\een
\end{cor}

 Of these nine cases, we believe all are new except for three. One is the signature $(-,\sqcup)$: see \cite{minusover} for a proof of completness of an equivalent set of axioms for the equational theory, and \cite{cirulis} and \cite{CLS} for full completeness proofs of a different set of equivalent axioms.  Another is the signature $(-,\sqcup,\cap)$: see the discussion below on comparison algebras and semigroups, and how they relate to pointed discriminator varieties. 
The third is the signature $(-,\cap)$: as discussed earlier, this signature is equivalent to the signature $(\circ,\backslash)$, which is finitely axiomatised in \cite{borlido}.

In a minus-algebra $(A,-)$, the relation $\sim$ actually respects minus as well: under the available functional interpretation, $a\sim b$ asserts that $a,b$ have the same domains, so it follows that $a\sim b$ and $c\sim d$ imply that $a-c\sim b-d$.  Given the functional interpretation, it is easy to check that $(A,-)/{\sim}$ is the dual of an implication algebra in the sense of \cite{implic}, hence is an implicative BCK-algebra by \cite{meng}.

Note that in the minus-algebra with override $(A,\circ,-,\sqcup)$, ${A/{\sim}}$ is a distributive lattice with operations induced by $\circ,\sqcup$ on $A$ (as follows easily from the functional interpretation), and then $F$ is a prime filter in a minus-algebra with override precisely when $F/{\sim}$ is a prime filter of this distributive lattice in the usual sense for distributive lattices.

Finally, the next result appears in Section $3$ of \cite{LeechNSL}. 

\begin{pro}  \label{Leechax}
The algebras of partial functions of signature $(\circ,\sqcup)$ are axiomatised as the class of right handed strongly distributive skew lattices.
\end{pro}

\section{Tidying up the axioms}  \label{sec:axioms2}

We now spend time re-considering the sets of axioms given in Section \ref{sec:axioms}, which were purpose-built to facilitate the subsequent completeness proofs in terms of partial functions.  For some of the quasi-equational axiomatisations it is possible to find an equivalent equational one or else show that this is impossible.  We also relate some of our axioms to other known axiomatisations.  (We defer consideration of the most parsimonious possible axiomatisations to the supplementary Section~\ref{sec:supp}.) 

First note that the operation of ``intersection" is abstractly defined in a skew Boolean intersection algebra (SBIA) in the sense of \cite{bigleech} in terms of the property that any two elements in a right normal band have a greatest lower bound under the natural order.  It turns out that this notion is relevant also in the case of right normal bands with intersection.

\begin{pro}  \label{nicemeet}
If $(A,\circ)$ is a right normal band in which every two elements $s,t\in A$ have a greatest lower bound $s\wedge t$ under the natural order, then $(A,\circ,\wedge)$ is a right normal band with intersection.  Moreover ever right normal band with intersection arises in this way.  

Hence the class of right normal bands with intersection $(A,\circ,\cap)$ may be axiomatised as follows.
\ben
\item $(A,\circ)$ is a right normal band;
\item  the following law holds: $$x=x\circ y \And x=x\circ z \Leftrightarrow x=x\circ(y\cap z).$$
\een
\end{pro}
\begin{proof}  Recall that $\leq$ is the natural order on $A$, given by $x\leq y$ if and only if $x=x\circ y$.  Note that for all $x,y\in A$, $x\circ y\leq y$, and if $y\leq z$ then $x\circ y\leq x\circ z$.  (Direct proofs are easy but also follow easily from the available functional interpetation.)

Assume  $(A,\circ)$ is a right normal band in which every two elements $s,t\in A$ have a greatest lower bound $s\wedge t$ under the natural order.  Then of course $(A,\wedge)$ is a semilattice, and by definition of the natural order it satisfies the law $(x\wedge y)\circ x=x\wedge y$.  It remains to check the law $x\circ(y\wedge z)=(x\circ y) \wedge z$.  

But for all $x,y,z\in A$, $x\circ (y\wedge z)\circ (x\circ y)=x\circ (y\wedge z)\circ y=x\circ (y\wedge z)$, so $x\circ (y\wedge z)\leq x\circ y$.  Moreover $x\circ (y\wedge z)\leq y\wedge z\leq z$.  So $x\circ (y\wedge z)\leq (x\circ y)\wedge z$.  

For the opposite inequality, since $x\circ y\leq y$, we have $(x\circ y)\wedge z\leq y\wedge z$, and so $x\circ ((x\circ y)\wedge z)\leq x\circ(y\wedge z)$.  Now suppose that $u\leq x\circ y$; then $x\circ u=x\circ u\circ x\circ y=u\circ x\circ x\circ y=u\circ x\circ y=u$.  Letting $u=(x\circ y)\wedge z$ gives that 
$$(x\circ y)\wedge z=x\circ((x\circ y)\wedge z)\leq x\circ (y\wedge z)\leq (x\circ y)\wedge z,$$
and the law follows.

Conversely, if $(A,\circ,\cap)$ is a right normal band with intersection, then $(x\cap y)\circ x=x$ so $x\cap y\leq x$, and by symmmetry (since $x\cap y=y\cap x$), $x\cap y\leq y$.  If $u\leq x,y$ then $u\circ (x\cap y)=(u\circ x)\cap(u\circ y)=x\cap y$, so $u\leq x\cap y$.  So $\cap$ is meet in $(A,\leq)$.
\end{proof}

Moving on to minus, recall that one of the laws of minus-algebras is an equational implication.  In fact it cannot be replaced by any purely equational laws.

\begin{pro} \label{properqe}
The class of minus-algebras is properly quasi-equational.  
\end{pro}
\begin{proof}
We give an example of a minus-algebra that has a quotient that is not a minus-algebra: specifically we show that necessary law
\begin{equation}
(x-y)\circ s=(x-y)\circ t \And y\circ s=y\circ t\Rightarrow x\circ s=x\circ t\label{eq:leftunionminus}
\end{equation}
 fails.  The example is essentially that used in the proof of \cite[Proposition 11]{modrest} to show that the class of modal restriction semigroups is a proper quasivariety.  We give the details as they are brief, and the law is different to the direct translation of the property considered in \cite{modrest}.
We consider the  minus-algebra $A$ of functions on $\{1,2,3\}$ consisting of the empty function along with the following six functions
\begin{align*}
a:=
\left(\begin{matrix}
1&2&3\\
1&2&2
\end{matrix}\right)&&
b:=
\left(\begin{matrix}
1&2&3\\
1&3&3
\end{matrix}\right)&&
c:= \left(\begin{matrix}
2&3\\
2&3
\end{matrix}\right)\\
d:=
\left(\begin{matrix}
1\\
1
\end{matrix}\right)&&
e:=
\left(\begin{matrix}
2&3\\
3&3
\end{matrix}\right)&&
f:= \left(\begin{matrix}
2&3\\
2&2
\end{matrix}\right)
\end{align*}
It is routine to verify that this is closed under minus. 
Observe that 
\begin{align}
(c-d)\circ a &=c\circ a=a-(a-c)=a-d=e\text{ and}\label{eqn:c-da}\\
(c-d)\circ b &=c\circ b=b-(b-c)=b-d=f.\label{eqn:c-db}
\end{align}
It is also easily verified that 
\begin{align}
d\circ a=d\circ b=d\label{eqn:da=db}
\end{align}
Moreover, the equivalence relation $\theta$ that identifies $e$ and $f$ is a congruence.  Indeed, for every $x$ we have $e-x$ is equivalent modulo $\theta$ to $f-x$, and similarly for $x-e$ and $x-f$.  As $\{e,f\}$ is the unique nontrivial block of $\theta$, this verifies the stability of $\theta$ under $-$.  In the quotient $A/\theta$ we have $((c-d)\circ a)/\theta=e/\theta=f/\theta((c-d)\circ b)/\theta$ by \eqref{eqn:c-da} and \eqref{eqn:c-db}.  By \eqref{eqn:da=db} we have $d\circ a=d\circ b$.  But $(c\circ a,c\circ b)=(a,b)\notin\theta$, so that the law \eqref{eq:leftunionminus} fails when $(x,y,s,t)=(c,d,a,b)$.
\end{proof} 

Proposition \ref{properqe} extends to cover a corresponding result for the class of minus-semigroups, as defining $st=0$ (the empty function) for all $s,t$ in the above example gives a minus-semigroup with $\theta$ still a congruence on it.

As already noted, the class of minus-algebras with intersection is a finitely based variety; this follows from the main result of \cite{borlido}, which is that the class of algebras of functions of signature $(\circ,\backslash)$ is a finitely axiomatised variety, given by the following laws:
\bi
\item $x\backslash (y\backslash x)=x$;
\item $x\cap y=y\cap x$;
\item $x\backslash y)\backslash z=(x\backslash z)\backslash y$;
\item $(x\circ z)\cap (y\circ z)=(x\backslash y)\backslash z$;
\item $(x\cap y)\circ x=x\cap y$.
\ei
(Here, as usual $s\cap t$ is defined to be $s\backslash(s\backslash t)$.)  This implies a finite equational axiomatisation for the equivalent signature $(-,\cap)$.  However, staying with the signature of $(\circ,\backslash)$, we are easily able to extend the above result appearing in \cite{borlido} to the signature in which composition is added.

\begin{pro}  \label{minusintvar}
The class of algebras of functions of signature $(\cdot,\circ,\backslash)$ is a finitely axiomatised variety, given by the following laws:
\bi
\item the above laws for algebras of signature $(\circ,\backslash)$ as in \cite{borlido};
\item the laws for 1-stacks with intersection;
\item $s(t\backslash u)=st\backslash su$.
\ei
\end{pro}
\begin{proof}
The above laws are all sound for partial functions (for the final one, see \cite{scheindiff}).  Conversely, given an algebra $(S,\cdot,\circ,\backslash)$ satisfying these laws, define $s-t$ and $s\cap t$ in terms of $\circ,\backslash$ as in Section \ref{intro}.  Then all the laws for minus-algebras with intersection must hold, since we are assuming the complete laws involving $\circ,\backslash$ as in \cite{borlido}.  We verify the minus-semigroup law: for all $s,t,u$,
\bea
s(t-u)&=&s(t\backslash (u\circ t))\\
&=&st\backslash s(u\circ t)\\
&=&st\backslash (su\circ st)\\
&=&st-su,
\eea
as required.  So by (5) in Theorem \ref{corstack} applied to (2) there, $(S,\cdot,-,\cap)$ can be represented as an algebra of partial functions.  It only remains to check that $\backslash,\circ$ are correctly represented.  They will be providing $s\backslash t=s-(s\cap t)$ and $s\circ t-t-(t-s)$ (since as in Section \ref{intro}, this is how they are definable in terms of $-,\cap$ as operations on partial functions), with $s-t,s\cap t$ as defined in Section~\ref{intro}.  But this again follows from the completeness of the laws for $\backslash,\circ$ in \cite{borlido}.
\end{proof}

We also have the following new result.

\begin{pro}
The class of minus-algebras with override is a finitely based variety, obtained by replacing the quasiequational law for minus-algebras by the equational law $x=(y\circ x)\sqcup (x-y)$; hence so is the class of minus-semigroups with override.
\end{pro}
\begin{proof}
Observe that the new equational law is sound.  Conversely, if $s-x=t-x$ and $x\circ s=x\circ t$ then 
$$s=(x\circ s)\sqcup(s-x)=(x\circ t)\sqcup(t-x)=t,$$
as required.
\end{proof}

Right handed strongly distributive skew lattices as in Subsection \ref{overonly} provide an alternative route to the axiomatisation of minus-algebras with override, the class considered in \cite{minusover}.  This is because minus is abstractly definable from override and domain restriction in the sense used in \cite[\S3.1]{jacstoverup}.

\begin{pro}  \label{overplusminus}
The class of algebras of functions of signature $(\circ,\sqcup,-)$ is axiomatised as the class of right handed strongly distributive skew lattices equipped with an operation  $-$ satisfying the following:
\ben
\item $0\circ x=0$;
\item $(x-y)\circ y=0$;
\item $(y\circ x)\sqcup(x-y)=x$.
\een
\end{pro}
\begin{proof} The three laws are easily seen to be sound, so it suffices to show that a right handed strongly distributive skew lattice of functions $(A,\circ,\sqcup)$ equipped with an operation $-$ satisfying the additional three laws for minus just given must have $-$ equal to minus.  Represent $(A,\circ,\sqcup)$ as an algebra of functions.  The law $0\circ x=0$ forces $0$ to be a subset of every function.  Modify the representation by removing this common part of every function: the result is easily seen to still be a faithful representation of $A$, but now $0$ is represented as the empty function.  The second law then guarantees that under this representation, the domain of $x-y$ lies outside that of $y$, and the third forces it to be minus applied to $x,y$.
\end{proof}

This provides a nice counterpoint to the previously noted fact that override is abstractly definable from minus.  In this case, the abstract definability of minus is from override, domain restriction and~$0$.

The axiomatisation of the signature $(\circ,\sqcup,-)$ obtained by Leech in \cite{LeechSBA} is very similar.   A  {\em right-handed  skew Boolean algebra} is an algebra $(A,\circ,\sqcup,-)$ satisfying the following laws:   
\bi
\item $(A,\circ,\sqcup)$ is a  right-handed strongly distributive skew lattice;
\item $x\circ 0=0\circ x=0$, $x\sqcup 0=0\sqcup x=x$;
\item $(x-y)\sqcup (y\circ x)=(y\circ x)\sqcup (x-y)=x$;
\item $(x-y)\circ y\circ x=y\circ x\circ (x-y)=0$.
\ei
In \cite{LeechSBA},  the author showed that the algebras of partial functions of signature $(\circ,\sqcup,-)$ are axiomatised as the class of right-handed skew Boolean algebras.  It follows that the axioms given in Proposition~\ref{overplusminus} for minus-algebras with override must be equivalent to the rather more complex axioms for right-handed skew Boolean algebras (something we have verified directly using {\em Prover9}).  The result in~\cite{LeechSBA} was published prior to the axiomatisation of the functional signature $(\circ,\sqcup)$ as right normal distributive symmetric skew lattices given in~\cite{LeechNSL} by the same author; we have just shown that one can easily obtain the richer axiomatisation involving minus from the one without it.

It follows from the third part of Corollary \ref{coralgebra} that the axioms just given are equivalent to those given earlier for minus-algebras with override, and so we may use the former in place of the latter to obtain an alternative axiomatisation for 1-stacks with minus and override.

Because adding intersection to the signature of minus-algebras with override requires only the addition of two equational laws as in Subsection \ref{cap}, we obtain the following.

\begin{cor}
The class of minus-algebras with intersection and override is a variety; hence so is the class of minus-semigroups with intersection and override.
\end{cor} 




\begin{pro}   \label{minusupvar}
The class of minus-algebras with update is a finitely based variety; hence so is the class of minus-semigroups with update.
\end{pro}
\begin{proof}
It suffices to show that the one quasiequation in the definition of minus-algebras follows from finitely many sound equations (which could therefore be added to that axiomatisation in place of it to give a finite equational axiomatisation).

Suppose $(A,\circ,-,\diamond)$ is a minus-algebra with update; hence it is functionally representable.  Suppose $a,b,c\in A$ are such that 
$c-b=a-b$ and $b\circ c=b\circ a$.  Our goal is to show that $a=c$.

Then using the sound law $x-y=(x-(y-z))-y$, we obtain 
$$c-a=(c-(a-b))-a=(c-(c-b))-a=(b\circ c)-a=b\circ a-a=0,$$
on using the sound law $x\circ y-y=0$.
So using the sound laws $x\circ y=(x\diamond y)-(x-y)$ and $x-0=x$, we obtain 
$$c\circ a=(c\diamond a)-(c-a)=(c\diamond a)-0=c\diamond a.$$

Now using the law $(x-y)\circ z=(x\circ z)-y$ for minus-algebras and again using the law $(x\circ y)-y=0$, we have 
$$(b-c)\circ a=(b\circ a)-c=(b\circ c)-c=0.$$  Hence $a-(b-c)=a-((b-c)\circ a))=a-0=a,$
upon using the further sound law $x-y=x-(y\circ x)$.
So again using the law $x-y=(x-(y-z))-y$, we see that  
$$a-c=(a-(c-b))-c=(a-(a-b))-c=(b\circ a)-c=0$$
from earlier.  So $c\diamond a=c\circ a=a-(a-c)=a-0=a$.

Now we note the following laws are sound:
$$((y\diamond x)\diamond(y-z))\diamond(z\circ y)=y,\ x\diamond (x-y)=x.$$
(The first is, because $(y\diamond x)\diamond(y-z)$ has domain the same as $y$, and agrees with $y$ outside of the domain of $z$, so updating with $y$ on the domain of $z$ forces it to equal $y$.  The second is obvious.)
Hence, $$c=((c\diamond a)\diamond (a-b))\diamond (b\circ a)=(a\diamond (a-b))\diamond (b\circ a)=a\diamond (b\circ a)
=a\diamond (a-(a-b))=a,$$
as required.
\end{proof}

\begin{cor}
The class of minus-algebras with intersection and update is a finitely based variety.
\end{cor}

For the signature consisting of minus, override and intersection, we saw earlier that it was trivial to add $\cap$ to the signature of minus-algebras with override using  finitely many more equations as in Subsection \ref{cap}, yielding a finitely based variety with three binary operations.  Alternatively, we can use Proposition \ref{nicemeet} and simply add the law 
$$x=x\circ y \And x=x\circ z \Leftrightarrow x=x\circ(y\circ z).$$
Now recalling that the class of minus-algebras with override is the same as the class of right handed skew Boolean algebras, this last observation leads us to infer that the class of minus-algebras with intersection and override is nothing but the class of {\em right handed skew Boolean intersection algebras} in the sense of \cite{bigleech}, which as noted in \cite{CLS} is also the same as the class of associative Boolean NLOs as defined by Cirulis in \cite{cirulis}.

From the remarks made in the first section, any of these axioms for the signature $(-,\sqcup,\cap)$ can then give a finite axiomatisation of the algebras of signature $(\backslash,\sqcup)$, which as noted above is a natural enrichment of the signature $(-,\sqcup)$ considered in \cite{minusover} and is the richest composition-free signature considered here.  This richest signature may usefully be thought of in yet another equivalent way.  

The operation of {\em generalised comparison} for partial functions was defined in \cite{compsemi}, (at least for the case $Y=X$ but the definition is identical in the general case).  Its definition is as follows: for $f,g,h,k\in P(X,Y)$ and $x\in X$,
\[
(f,g)[h,k](x):=\begin{cases} h(x)&\mbox{ if }f(x)=g(x)\mbox{ or neither is defined}\\
k(x)&\mbox{ otherwise.}
\end{cases}
\]
In short, $(f,g)[h,k]$ is $h$ when $f,g$ do not disagree and $k$ otherwise.  As follows easily from what is observed there for the special case in which $Y=X$, it is easy enough to see that $$f\backslash g=(f,f\cap g)[0,f]\mbox{ and }f\sqcup g=(f,g)[g,f].$$
Indeed we may explicitly write
\bi
\item $f\cap g=(f,g)[f,0]$
\item $f\circ g=(f,0)[0,g]$
\item $g-f=(f,0)[g,0]$.
\ei
Conversely, it is straightforward to verify that 
$$(f,g)[h,k]=(f\cap g)\circ h\sqcup ((h-f)-g)\sqcup k\mbox{ and }0=f-f,$$
so the signature consisting of generalised comparison and zero is also equivalent to $(\backslash,\sqcup)$ and hence to $(-,\cap,\sqcup)$.  A functional representation of generalised comparison semigroups with zero was given in \cite{compsemi}, which therefore provides a prior proof of the special case case of \eqref{minusmore} in Theorem \ref{main2} applied to \eqref{overmain}, and similarly for the corresponding parts of Corollary \ref{coralgebra}.

On this topic, it is shown in \cite{bigleech} that the variety $PD_0$ generated by the so-called pointed discriminator algebras is equivalent to the variety of skew Boolean intersection algebras.  A special case of this gives that the class of right handed skew Boolean intersection algebras (hence also of these various other classes) is term equivalent to the variety of generalised comparison algebras with zero as in \cite{compsemi}, something which follows immediately from the above remarks (since both are term equivalent to the variety of minus-algebras with intersection and override).
  
We turn to the remaining signatures, which do not include minus.  

As we have seen, the class of minus-semigroups with intersection and override axiomatises the algebras of functions under the signature $(-,\sqcup,\cap)$, and in fact these axioms are equivalent to those of right handed skew Boolean intersection algebras.  This latter class consists of those right handed skew Boolean algebras in which every two elements have a meet $\cap$ under the natural order.  This follows from the fact that the class of right normal bands with intersection is nothing but the class of right normal bands such that every two elements have a meet under the natural order.  Since partial functions under signature $(\circ,\sqcup)$ are axiomatisable as the class of right handed strongly distributive skew lattices, one might imagine by analogy that the class of right normal bands with intersection and override may be axiomatisable as those right handed strongly distributive skew lattices in which every two elements have a meet under the natural order.  But this is not the case.


\begin{eg}\label{eg:droi}
Consider the set of partial functions $A=\{1,i,e,f,0\}$ on $X=\{a,b\}$ defined as follows: 
\[
1=\{(a,a),(b,b)\},\ i=\{(a,b),(b,b)\},\ e=\{(a,a)\},\ f=\{(a,b)\},\ 0=\varnothing.
\]
It is easily checked that $A$ is closed under domain restriction and override, so is a right handed strongly distributive skew lattice. Further, any two elements $x,y\in A$ have a least upper bound $x\wedge y$ under the natural order.  
However, 
$$((e\sqcup i)\wedge i)\sqcup e=(1\wedge i)\sqcup e=0\sqcup e=e\neq 1=e\sqcup i,$$
so the axiom $((x\sqcup y)\cap y)\sqcup x=x\sqcup y$ for right normal bands with intersection and override does not hold.
\end{eg}

%

We remark that this example is a right normal band having meet with respect to its natural order, so by Proposition \ref{nicemeet}, it must be faithfully representable in terms of partial functions as such.  This may be achieved by adjusting the definition of its elements so that $i$ becomes $\{(a,b),(b,a)\}$ with all others unchanged.  However, once this is done, $\sqcup$ is no longer correctly represented since $e\sqcup i$ no longer equals $1$.

One remedy is to add to the laws for right-handed strongly distributive skew lattices all of the laws for right normal bands with intersection and override as in Subsection \ref{sub:intover} (which by the sixth part of Corollary \ref{coralgebra} are complete for this signature), although the result contains repetitions and probably further redundancy.  These redundancies are explored in the supplementary material in Section~\ref{sec:supp}.

The class of right normal bands with intersection and update includes a quasiequation in its axiomatic definition.  This cannot be replaced by equations.

\begin{pro}  \label{updatecapqv}
The class of right normal bands with intersection and update is properly quasi-equational, and so is the class of 1-stacks with intersection and update.
\end{pro}
\begin{proof}
In $P(X_1,X_2)$, where $X_1=\{x,y\}$ and $X_2=\{x',y'\}$, with all four of $x,y,x',y'$ distinct, let $I$ denote the function $\{(x,x'),(y,y')\}$, $b=\{(y,x')\}$, $c=\{(y,y')\}$, $d=\{(x,x'),(y,x')\}$, $e=\{(x,x')\}$ and let $0$ denote the empty function.  Let $A=\{1,b,c,d,e,0\}$.  It is routine to check that $A$ is closed under $\circ,\diamond,\cap$ and is therefore a right normal band with intersection and update.  It is even closed under composition, with all composites equalling $0$, and so is a 1-stack with intersection and update.    The various Cayley tables for the operations (other than composition) are as follows.

\begin{center}
\begin{tabular}{c|cccccc}
$\circ$&$1$&$b$&$c$&$d$&$e$&$0$\\
\hline
$1$&$1$&$b$&$c$&$d$&$e$&$0$\\
$b$&$c$&$b$&$c$&$b$&$0$&$0$\\
$c$&$c$&$b$&$c$&$b$&$0$&$0$\\
$d$&$a$&$b$&$c$&$d$&$e$&$0$\\
$e$&$e$&$0$&$0$&$e$&$e$&$0$\\
$0$&$0$&$0$&$0$&$0$&$0$&$0$
\end{tabular}
\hspace{3cm}
\begin{tabular}{c|cccccc}
$\diamond$&$1$&$b$&$c$&$d$&$e$&$0$\\
\hline
$1$&$1$&$d$&$1$&$d$&$1$&$1$\\
$b$&$c$&$b$&$c$&$b$&$b$&$b$\\
$c$&$c$&$b$&$c$&$b$&$c$&$c$\\
$d$&$a$&$d$&$a$&$d$&$d$&$d$\\
$e$&$e$&$e$&$e$&$e$&$e$&$e$\\
$0$&$0$&$0$&$0$&$0$&$0$&$0$
\end{tabular}
\end{center}
\medskip

\begin{center}
\begin{tabular}{c|cccccc}
$\cap$&$1$&$b$&$c$&$d$&$e$&$0$\\
\hline
$1$&$1$&$0$&$c$&$e$&$e$&$0$\\
$b$&$0$&$b$&$0$&$b$&$0$&$0$\\
$c$&$c$&$0$&$c$&$0$&$0$&$0$\\
$d$&$e$&$b$&$0$&$d$&$e$&$0$\\
$e$&$e$&$0$&$0$&$e$&$e$&$0$\\
$0$&$0$&$0$&$0$&$0$&$0$&$0$
\end{tabular}
\end{center}
\medskip

\noindent
From these, it is easy to see that the equivalence relation $\theta$ in which $(e,0)\in \theta$, with all other elements of $A$ in their own $\theta$-classes, is a congruence with respect to all of these three operations, and trivially with respect to composition as well.  However, in the quotient algebra $A/\theta$ equipped with the induced operations, the one quasiequation in the definition of right normal bands with update fails.  This is because 
$$(a\cap(a\diamond b))\circ a=(a\cap d)\circ a=e\circ a=e$$
and
$$(a\cap(a\diamond b)))\circ c=e\circ c=0\ \theta\ e,$$
and moreover $b\circ a=c=b\circ c$.  However, 
$a\circ a=a$ while $b\circ a=c$ and $(a,c)\not\in \theta$.
\end{proof}

Here is a table summarizing what is known.  In it, if the status of a case is known, there is an  axiomatisation witnessing that status that is finitely based. The cases shown here for the first time (as far as we know) are indicated with a ``$*$".
\bigskip

\begin{tabular}{|c|c|c|}
\hline
signature&composition-free&with composition\\
\hline
$(\circ)$&variety (right normal bands)&variety (1-stacks)\\
$(\circ,\cap)$&variety $*$&variety $*$\\
$(-)$&proper quasivariety $*$&proper quasivariety $*$\\
$(-,\cap)$&variety  (\cite{borlido})& variety $*$\\
$(-,\sqcup)$&variety (RH SBAs)&variety $*$\\
$(-,\sqcup,\cap)$&variety (RH SBIAs)&variety (comparison semis with zero)\\
$(-,\diamond)$&variety $*$&variety $*$\\
$(-,\diamond,\cap)$&variety $*$&variety $*$\\
$(\circ,\sqcup)$&variety (RH strongly dist. skew lattices)&unknown\\
$(\circ,\diamond)$&unknown&unknown\\
$(\circ,\sqcup,\cap)$&variety $*$&variety $*$\\
$(\circ,\diamond,\cap)$&proper quasivariety $*$&proper quasivariety $*$\\
\hline
\end{tabular}
\bigskip

\section{Left restriction semigroups and their enrichments}\label{sec:restriction}

In this section we connect the classes of algebras considered here with left restriction semigroups and their enrichments.  This material only applies to the cases in which composition is being modelled, so at least 1-stacks.

A {\em left restriction semigroup} is a unary semigroup with unary operation $D$ satisfying the following laws:
\bi
\item $D(x)x=x$;
\item $D(x) D(y)=D(y)D(x)$;
\item $D(D(x)y)=D(D(x)D(y))$;
\item  $xD(y)=D(xy)x$.
\ei
It is well-known that the laws for left restriction semigroups axiomatise the subalgebras of $PT(X)$ under composition and domain, defined as follows:
$$D(f)=\{(x,x)\in X\times X\mid x\in dom(f)\}.$$
This was first shown by \cite{trok}, albeit using slightly different axioms and terminology.

The following consequences of these laws are easily shown and in any case well-known.

\begin{lem}\label{lem:extralaws}
If $(A,\cdot,D)$ is a left restriction semigroup, then for $x,y\in A$,
\ben
\item $D(x)^2=D(x)$,
\item $D(x)D(y)=D(y)D(x)=D(D(x)y)$,
\item $D(xy)D(x)=D(xy)$,
\item $D(xy)=D(xD(y))$.
\een 
\end{lem}

It is well-known, and in any case follows from the well-known representation theorem for left restriction semigroups, that every left restriction semigroup $(A,\cdot, D)$ gives a 1-stack $(A,\cdot,\circ)$ if we define $s\circ t=D(s)t$; call this the {\em derived 1-stack} of the left restriction semigroup.  However, in general 1-stacks are not equivalent to left restriction semigroups, since there are 1-stacks of functions containing no restrictions of the identity map.  

One situation in which 1-stacks and left restriction semigroups do correspond arises when the operation modelling composition in a 1-stack has an identity element $1$, for then we may define $D(s)=s\circ 1$ and obtain a left restriction semigroup from which the original 1-stack may be derived.  In fact, if an identity is present, every type of enriched 1-stack so far considered can be viewed as being a type of enriched left restriction semigroup with identity.  By modifying the representation~$\Phi$ used in Proposition \ref{main} by not adjoining a new identity element $e$ to $A$ and making other small tweaks, it is possible to ensure the resulting $\Phi$ is still a faithful representation, but one in which~$1$ is represented as the identity function.  We could then go on to obtain results for enriched left restriction monoids corresponding to our earlier results for enriched 1-stacks.

We note that each of the operations considered here can be expressed in the language of the modal restriction semigroups considered in \cite{modrest}.  There, algebras of partial functions in $PT(X)$ equipped with at least composition and ``antidomain" $A$, where $A(t)$ is the identity map on $X$ restricted to the complement of $\dom(t)$, $t\in PT(X)$ together with some or all of intersection and override (there called ``preferential union") in all possible combinations, were axiomatised.   Thus:
$$D(s):=A(A(s)),\ s\circ t:=D(s)t,\ s-t:=A(t)s,\ s\diamond t:=D(s)(t\sqcup s).$$
Conversely, the signature $(\cdot ,1,A)$, where $1$ is the identity function, is easily seen to be equivalent to $(\cdot,1,-)$, since $A(s)=1-s$ for all $s\in PT(X)$, so modal restriction semigroups are nothing but minus-semigroups with identity.  This then extends to the various other signatures.   The above fairly straightforward adjustment to the definition of $\Phi$ in Proposition \ref{main} does represent $1\in A$ correctly, and all else still works, so we could thereby recover the axiomatisations obtained in \cite{modrest}.

However, the signatures considered in this article do not include the identity function, so axiomatiations for algebras different to those considered in \cite{modrest} may be obtained.  Instead, we start with an arbitrary left restriction semigroup $(A,\cdot, D)$, possibly without identity, and represent its derived 1-stack using $\Phi$ as in Proposition \ref{main}.

\begin{pro}  \label{lrs}
Let $(A,\cdot,D)$ be a left restriction semigroup with $(A,\cdot,\circ)$ the derived 1-stack, and suppose ${\mc F}$ is any separating set of filters of $A$.  Then $\Phi$ respects $D$.
\end{pro}
\begin{proof}  
Now $\ol{x}_F\in \dom(\Phi(D(a)))$ if and only if $xD(a)\in F$, that is, $D(xa)x\in F$, or $(xa)\circ x\in F$, that is, $xa\in F$, or $\ol{x}^F\in \dom\Phi(a)$.  For such $\ol{x}^F$,  note that $(xa)\circ x=D(xa)x=xD(a)=D(xD(a))xD(a)=D(xa)xD(a)=(xa)\circ (xD(a))$, so $(x,xD(a))\in \epsilon_F$.  Hence $\Phi(D(a))(\ol{x}^F)=\ol{xD(a)}^F=\ol{x}^F$, and so $\Phi(D(a))$ is the identity function on $\dom(\Phi(a))$, that is,  $\Phi(D(a))=D(\Phi(a))$.
\end{proof}

It now follows that we can obtain further axiomatisations of algebras of partial functions, one for each of those considered so far, in which $D$ is part of the signature.  To obtain the new axioms, simply include the left restriction semigroup axioms and interpret every occurrence of ``$s\circ t$" in an axiom  as ``$D(s)t$".  In this way, we immediately recover the familiar fact that left restriction semigroups axiomatise algebras of partial functions under composition and domain, and can go on to recover some previously known axiomatisations as well as to uncover some new ones.

Here is a summary of the results applied to enriched left restriction semigroups in the way just described.  In each case, axioms follow via the axioms for the relevant class of enriched 1-stacks given in Section \ref{sec:axioms}, modified to accommodate $D$ as above, and using the relevant part of Theorem~\ref{corstack} as well as Proposition \ref{lrs} above.

\ben
\item Signature $(\cdot,D,\cap)$.  The resulting variety of enriched left restriction semigroups was first axiomatised in \cite{agreeable} and \cite{dudtro}.  
\item Signature $(\cdot,D,-)$.  These have not previously been considered.  Note that the class of minus-semigroups is a proper quasivariety, as noted earlier.  When we introduce left restriction semigroup structure to induce the right normal band structure, then adding an identity element $1$ with respect to the composition operation, we obtain the class of modal restriction semigroups in which $A(x)=1-x$, as in \cite{modrest}.  These form a proper quasivariety, as shown there, so the current class must also.
\item Signature $(\cdot,D,-,\cap)$.  Not previously considered, but a variety, following Proposition \ref{minusintvar}.  With an identity element under composition added, we obtain modal restriction semigroups with intersection as in \cite{modrest}.
\item Signature $(\cdot,D,-,\sqcup)$.  Again, not previously considered, but yields a finitely based variety, as did $(\cdot,\circ,-,\sqcup)$.  With an identity element under composition added, we obtain modal restriction semigroups with preferential union as in \cite{modrest}.
\item Signature $(\cdot,D,-,\diamond)$.  This is a new signature not previously considered, even with an identity element present, and is a finitely based variety due to Proposition \ref{minusupvar}.  With a composition identity added, the resulting class of modal restriction semigroups is a new one enriched with the addition of update. 
\item Signature $(\cdot,D,\cap,\sqcup)$.  This is a signature not previously considered, again a finitely based variety.  Since minus is not present, adding a composition identity does not give a class of enriched modal restriction semigroups.  
\item Signature $(\cdot,D,\cap,\diamond)$.  Again, a new signature.  We are uncertain whether it is a variety.
\een

\section{Open questions}  \label{sec:open}

This work considers signatures that can express domain restriction of functions $\circ$.  As already noted, Leech showed in \cite{LeechNSL} that the functional algebras of signature $(\circ,\sqcup)$, or equivalently $(\circ,\sqcup,\diamond)$, are finitely axiomatised as the variety of right handed strongly distributive skew lattices.  However, despite the fact that the concept of a prime filter as in Subsection \ref{intover} makes sense in this setting, our representation method as in Proposition \ref{main} does not seem to represent $\sqcup$ correctly when ${\mc F}$ is chosen to be the prime filters.  Hence the extended version of this signature in which composition is present remains unaxiomatised.  

We showed in Proposition \ref{updatecapqv} that the signature $(\circ,\diamond,\cap)$, with and without composition added, gives a properly quasiequational axiomatisation.  However, although it seems likely, we are yet to confirm that the class of functional algebras of signature  $(\cdot,D,-,\diamond)$ is a proper quasivariety.  (The example in the proof of Proposition \ref{updatecapqv} does not readily adapt to this setting.)


The set $P(X,Y)$ has a relational generalisation $R(X,Y)$, consisting of all binary relations that are subsets of the power set of $X\times Y$.  The various operations described in this work all generalise to $R(X,Y)$ (within which $P(X,Y)$ is a subalgebra).  There is interest in axiomatising algebras of binary relations under these operations.  It is noted in Remark $3.5$ of \cite{jacstoverup} that for signatures chosen from the operations $(\circ,\sqcup,-,\diamond)$, up to isomorphism, the relational models are precisely the same as the functional ones.  Hence the algebras of binary relations under such signatures have the same axiomatisations as the algebras of functions of the same signature.  However, when adding in either intersection or composition to these signatures, this is no longer the case in general.  There is interest in whether finite axiomatisations exist for such signatures, and if so what they are.

\newpage
\section{Supplementary material}\label{sec:supp}
There are many more simplifications to axioms that can be obtained, mostly with the aid of {\em Prover9/Mace4}, but as these distract somewhat from the flow of the main results of the article, and are often mechanical in nature, we collect them here as supplementary material.  Proofs obtained mechanically are omitted, but can be easily verified using {\em Prover9/Mace4} by entering the given laws and testing for derivation of the relevant axioms obtained in the main body of the paper.

In Proposition \ref{nicemeet} we gave a simplification of the axioms in Subsection \ref{cap} for right normal bands with intersection.  Using {\em Prover9} and {\em Mace4}, it is straightforward to verify the following significant trimming of the eight equations implicit in the definition of right normal bands with intersection.
\begin{pro}  \label{capslick}
An irredundant equational axiomatisation for the class of right normal bands with intersection is the following.
\bi
\item $(x\circ y)\circ z=(y\circ x)\circ z$
\item $x\cap x=x$
\item $x\cap y=y\cap x$
\item $(x\cap y)\circ x=x\cap y$ 
\item $x\circ (y\cap z)=(x\circ y)\cap z$
\ei
\end{pro}

Use of {\em Prover9} establishes that many of the laws  for minus-algebras given in Subsection \ref{subsec:minus} are redundant.  We omit the rather long human proof, and instead present a sufficient irredundant set of axioms in the next result. 

\begin{pro}  \label{slickminus}
The axioms for minus-algebras are equivalent to the following ones.
\bi
\item $(x\circ y)\circ z=(y\circ x)\circ z$
\item $x\circ y=y-(y-x)$
\item $0\circ x=0$
\item $(x-y)\circ z=(x\circ z)-y$
\item $s-x=t-x\And x\circ s =x\circ t\Rightarrow s=t$.
\ei
\end{pro}

We note that if one uses the streamlined axioms for minus-algebras given in Proposition \ref{slickminus}, just adding the two for override as in Subsection \ref{sub:intover}, and replacing the quasiequation for minus-algebras by the equational law $x=(y\circ x)\sqcup (x-y)$, we find (at length) that the final right normal band law is redundant.   This gives the following.

\begin{pro}  \label{minusovertrick} 
The class of minus-algebras with override may be axiomatised as follows.
\bi
\item $x\circ y=y-(y-x)$ 
\item $0\circ x=0$  
\item $(x-y)\circ z=(x\circ z)-y$
\item $x=(y\circ x)\sqcup(x-y)$
\item $(x\sqcup y)-x=y-x$
\item $x\circ (x\sqcup y)=x$
\ei
\end{pro}

By eliminating $\circ$ using the first law, we obtain five laws, the same number given in \cite{minusover}, the latter an axiomatisation shown equationally complete there and subsequently shown complete in~\cite{CLS}.  
Introducing an operation modelling composition as in Subsection \ref{sub:minusover}, {\em Prover9} shows that associativity of composition is redundant, so only the law $x(y-z)=(xy)-(xz)$ need be added.  

The axioms of Proposition \ref{overplusminus} for minus-algebras with override, built from those for right handed strongly distributive skew lattices, also contain redundancies.  With the aid of {\em Prover9} we obtain the following simplification in which 14 axioms are reduced to five.

\begin{pro}  \label{minusovertrick2} 
The class of minus-algebras with override may be irredudantly axiomatised as follows.
\bi
\item $x\circ (y\sqcup z)=(x\circ y)\sqcup (x\circ z)$
\item $(x\sqcup y)\circ z=(x\circ z)\sqcup (y\circ z)$
\item $x\circ (x\sqcup y)=x$
\item $(x-y)\circ y=0$
\item $(y\circ x)\sqcup(x-y)=x$
\ei
\end{pro}

For the signature of domain restriction and override, starting with Leech's axioms for right handed strongly distributive skew lattices, {\em Prover9/Mace4} showed that the following somewhat less impressive paring back of axioms is possible.

\begin{pro}  \label{slickskew}
For the algebras of functions of signature $(\circ,\sqcup)$, the following axiomatisation is irredundant.
\bi
\item $(x\circ y)\circ x=y\circ x$
\item $(x\sqcup y)\sqcup z=x\sqcup (y\sqcup z)$
\item $x\circ (x\sqcup y)=x$
\item $(y\sqcup x)\circ x=x$
\item $x\circ(y\sqcup z)=(x\circ y)\sqcup (x\circ z)$
\item $(x\sqcup y)\circ z=(x\circ z)\sqcup (y\circ z)$
\ei
\end{pro}

In Example \ref{eg:droi} of the main text, we showed that a possibly simpler set of laws for the signature $\{\circ,\sqcup,\cap\}$ was not complete, and subsequently observed the required extra laws: add to the laws for right-handed strongly distributive skew lattices all of the laws for right normal bands with intersection and override as in Subsection \ref{sub:intover}.  It was noted that these laws would contain considerable redundancy including repetitions, and indeed {\em Prover9/Mace4} shows that of these additional laws added to those for right-handed strongly distributive skew lattices, it is sufficient to add only the law shown to fail in Example \ref{eg:droi}; moreover the laws for right-handed strongly distributive skew lattice given in Proposition \ref{slickskew} also simplify.

\begin{pro}
The class of right normal bands with intersection and override may be axiomatised as those algebras $(A,\circ,\sqcup,\cap)$ such that:
\bi
\item $x\circ (x\sqcup y)=x$
\item $(y\sqcup x)\circ x=x$
\item $x\circ(y\sqcup z)=(x\circ y)\sqcup (x\circ z)$
\item $(x\sqcup y)\circ z=(x\circ z)\sqcup (y\circ z)$
\item $x=x\circ y\And x=x\circ z \Leftrightarrow x=x\circ (y\cap z)$
\item $((x\sqcup y)\cap y)\sqcup x=x\sqcup y$
\ei
\end{pro}  

Starting instead with the axioms given in Subsection \ref{sub:intover}, {\em Prover9/Mace4} revealed that the following set of axioms is minimally complete; note that it contains none of the right normal band laws.  

\begin{pro}  \label{bestoverint}
For the class of right normal bands with intersection and override, the following axiomatisation is irredundant.
\bi
\item $x\cap y=y\cap x$
\item $x\circ (y\cap z)=(x\circ y)\cap z$
\item $x\circ(x\sqcup y)=x$
\item $((x\sqcup y)\cap y)\sqcup x=x\sqcup y$
\item $(x\sqcup y)\circ z=(x\circ z)\sqcup (y\circ z)$
\ei
\end{pro}

As discussed in the introduction, it is possible to entirely eliminate $\circ$ from the signature, instead replacing it by update, using the facts that $x\circ y=(x\diamond y)\cap y$ and $x\diamond y=x\circ (y\sqcup x)$, thereby giving an equational axiomatisation for the algebras of partial functions of signature $(\sqcup,\diamond,\cap)$.  We could do this directly for the above axioms.  However, instead it seemed of more interest to start from the axioms for override and update found in \cite{jacstoverup} and to attempt to augment them with enough of the above axioms to incorporate intersection into that axiomatisation.  Here are those axioms from \cite{jacstoverup}, proved complete there.

\bi
\item $x\sqcup (y\sqcup z)=(x\sqcup y)\sqcup z$
\item $x\sqcup x=x$
\item $x=x\diamond (x\sqcup y)$
\item $x\sqcup y=(y\diamond x)\sqcup x$
\item $(x\diamond y)\diamond z=x\diamond (z\sqcup y)$
\item $(x\sqcup y)\diamond  z=(x\diamond z)\sqcup (y\diamond z)$
\ei

We then added translated versions of the laws in Proposition \ref{bestoverint} to these, eliminating as many as possible relative to those in \cite{jacstoverup}, and only then turning to elimination of any redundant laws in the above axioms for override and update, all done using {\em Prover9/Mace4}.  The result of this was that we were only able to eliminate one of our translated laws.  We also realised that, on translation, the third law above was equivalent to the combination of the two laws $x\diamond(x\sqcup y)=x$ and $x\cap(x\sqcup y)=x$, the former of which proved redundant.  With this noted, of the above laws for override and update, most then proved redundant, leaving us with the following.

\begin{pro}
An irredundant axiomatisation for the algebras of functions of signature $(\sqcup,\diamond,\cap)$ is as follows:
\bi
\item $x=x\diamond(x\sqcup y)$
\item $(x\diamond y)\diamond z=x\diamond(z\sqcup y)$
\item $x\cap y=y\cap x$
\item $(x\diamond (y\cap z))\cap(y\cap z)=((x\diamond y)\cap y)\cap z$ 
\item $x\cap (x\sqcup y)=x$
\item $((x\sqcup y)\cap y)\sqcup x=x\sqcup y$
\item $(x\sqcup y)\circ z=(x\circ z)\sqcup (y\circ z)$
\ei
\end{pro}


For the class of 1-stacks with intersection and override, one further axiom may be omitted, namely $x(y\circ z)=xy\circ xz$.

Next, for update and intersection, the software showed the following.  (Proving that the axiom  $(x\circ y)\circ z=(y\circ x)\circ z$ is redundant was very slow.) 

\begin{pro}
For the class of right normal bands with intersection and update, the following axiomatisation is irredundant.
\bi
\item $x\cap y=y\cap x$  
\item $(x\cap y)\circ x=x\cap y$
\item $x\circ (y\cap z)=(x\circ y)\cap z$
\item $(x\diamond y)\circ x=x$  
\item $x\diamond (x\diamond y)=x\diamond y$
\item $y\circ (x\diamond y)=x\circ y$
\item $(x\cap (x\diamond y))\circ u=(x\cap (x\diamond y))\circ v \And y\circ u=y\circ v \Rightarrow x\circ u=x\circ v$
\ei
\end{pro}

Again, it is possible to eliminate $\circ$ entirely, giving axioms involving $\diamond,\cap$ that axiomatise algebras of partial functions of this signature. 

Finally we turn our attention to the axioms of Section \ref{sec:restriction} in the main text.  \emph{Prover9/Mace4} showed the following two results.

\begin{pro}  
The following axiomatisation is irredundant for the functional signature $(\cdot,D,\cap)$:
\bi
\item $x(yz)=(xy)z$
\item $xD(y)=D(xy)x$
\item $x\cap x=x$
\item $x\cap y=y\cap x$
\item $D(x\cap y)x=x\cap y$
\item $D(x)(y\cap z)=(D(x)y)\cap z$
\item $x(y\cap z)=(xy)\cap (xz)$
\ei 
\end{pro}

\begin{pro}  
The following axiomatisation is irredundant for the signature $(\cdot,D,\cap,\sqcup)$:
\bi
\item $x(yz)=(xy)z$
\item $xD(y)=D(xy)x$
\item $x\cap y=y\cap x$
\item $D(x)(y\cap z)=(D(x)y)\cap z$
\item $x(y\cap z)=(xy)\cap (xz)$
\item $D(x)(x\sqcup y)=x$
\item $((x\sqcup y)\cap y)\sqcup x=x\sqcup y$
\item $D(x\sqcup y)z=(D(x)z)\sqcup D(y)z$
\ei 
\end{pro}

\end{document}